\documentclass[reqno]{amsart}
\usepackage{graphicx}
\usepackage{amsmath, amsthm, amssymb, dsfont}
\usepackage{thmtools}
\usepackage{mathtools}
\usepackage{centernot}
\usepackage{amsfonts}
\usepackage{geometry}
\usepackage{cancel}
\usepackage{placeins}
\usepackage{tikz}
\usepackage{pgfplots}
\pgfplotsset{compat=1.18}
\usepackage{MnSymbol}
\usepackage[parfill]{parskip}
\usepackage{comment}
\usepackage{hyperref}
\usepackage{cleveref} 
\usepackage{mathrsfs}
\usepackage{thm-restate}
\usepackage{esint}
\usepackage{enumitem}
\usepackage{tcolorbox}
\usepackage{float}
\usepackage{caption}
\usepackage{framed}
\usepackage{empheq}
\newtheorem{theorem}{Theorem}[section]
\newtheorem*{theorem*}{Theorem}
\newtheorem{lemma}[theorem]{Lemma}

\newtheorem{proposition}[theorem]{Proposition}
\newtheorem*{proposition*}{Proposition}

\newtheorem*{claim*}{Claim}
\newtheorem{corollary}[theorem]{Corollary}

\newtheorem{example}[theorem]{Example}
\theoremstyle{definition}
\newtheorem{definition}[theorem]{Definition}
\newtheorem{remark}[theorem]{Remark}

\newcommand{\sgn}{\operatorname{sgn}}

\newcommand{\thmendspace}{\vspace{-\belowdisplayskip}}
\newcommand{\norm}[1]{\left \lVert #1 \right \rVert}
\newcommand{\R}{\mathbb{R}}

\newcommand{\Q}{\mathbb{Q}}
\newcommand{\N}{\mathbb{N}}

\DeclareMathOperator\supp{supp}
\setlist[itemize,enumerate]{nosep}

\makeatletter
\renewcommand*\env@matrix[1][*\c@MaxMatrixCols c]{%
  \hskip -\arraycolsep
  \let\@ifnextchar\new@ifnextchar
  \array{#1}}
\makeatother

\begin{document}

\title[Existence for non-semiconvex interaction potentials]{On one-dimensional pressureless Euler--Poisson systems: existence for non-semiconvex interaction potentials}

\author[Xie]{Andrew Xie}
\address{Andrew Xie\\Rutgers University -- New Brunswick}
\email{ax52@scarletmail.rutgers.edu}
\author[Alayande]{Emmanuel Alayande}
\address{Emmanuel Alayande\\University of Florida}
\email{e.alayande@ufl.edu}
\author[Chun]{Christopher Chun}
\address{Christopher Chun\\The University of Texas at Austin}
\email{chrischun@utexas.edu}
\author[Li]{Yusen Li}
\address{Yusen Li\\Stony Brook University}
\email{yusen.li@stonybrook.edu}
\author[Mawalkar]{Rohan Mawalkar}
\address{Rohan Mawalkar\\The Ohio State University}
\email{mawalkar.5@osu.edu}
\author[Tudorascu]{Adrian Tudorascu}
\address{Adrian Tudorascu\\West Virginia University}
\email{adtudorascu@mail.wvu.edu}
\author[Van]{Adam Van}
\address{Adam Van\\Duke University}
\email{adam.van@duke.edu}



\begin{abstract}
One-dimensional Euler-Poisson-type systems with interaction potentials that are repulsive at short range and attractive at long range are useful simplified models that describe biological aggregation. In the current work we employ new techniques to treat compactly supported solutions in the case where the potential is $C^1$ and satisfies certain coercivity growth conditions although it is not necessarily semiconvex, as required in previous work. Central to this approach are uniform bounds on the positions and velocities of discrete sticky-particle solutions to such a system, which requires the potential be only $C^{0, 1}_{\mathrm{loc}}$. In the semiconvex case, these bounds also yield convergence of time-shifted trajectories to an eternal solution. The tradeoff, just as in the case of repulsive Euler-Poisson, is that we cannot claim that the emerging solutions for non-semiconvex potentials are sticky.
\end{abstract}

\maketitle

\section{Introduction}
\subsection{Overview} 
One-dimensional Euler-Poisson-type systems with interaction potentials that are repulsive at short range and attractive at long range are useful simplified models that describe biological aggregation. Over the last eight years or so there has been significant progress on existence of sticky particles solutions for systems with interaction potentials that are $C^1$ and $c$-convex, a feature which allows for constructing solutions which retain the ``stickiness" property from their discrete approximations. Here, $c$-convexity of a potential $W$ means that there exists a $c > 0$ such that $x \mapsto W(x) + \frac{c}{2} x^2$ is convex; we will use this interchangeably with semiconvex. The attractive Euler-Poisson system was also studied by approximating its potential uniformly with convex $C^1$ potentials. In the repulsive case such an approximation is unavailable, yet distributional solutions were obtained about two decades ago by Nguyen and Tudorascu via the Scalar Conservation Laws approach due to Brenier and Grenier. However, these solutions are not necessarily sticky solutions. Additionally, such techniques require an interaction potential whose force reduces to a function of the mass label alone and does not depend on the absolute distance between particles, whereas here we allow for potentials that depend on such distances. 

In the current work we employ new techniques to treat compactly supported solutions in the case where the potential is $C^1$ and satisfies certain coercivity growth conditions although it is not necessarily $c$-convex. Central to this approach are uniform bounds on the positions and velocities of discrete sticky-particle solutions to such a system. Notably, confinement requires less regularity of the potential than existence (only $C^{0, 1}_{\mathrm{loc}}$), and does not always require the attractive force to be stronger than the local repulsive force, see \Cref{rem:boundedforceconfinement}. These confinement results also yields additional tools in the semiconvex case, as developed in the Appendix. The tradeoff, just as in the case of repulsive Euler-Poisson, is that we cannot claim that the emerging solutions for non-semiconvex potentials are sticky.

This paper concerns the one-dimensional pressureless Euler--Poisson system with non-local interactions
\begin{subequations}\label{eq:PEP}
\begin{empheq}[left={\empheqlbrace\;\;}]{align}
\partial_t\rho+\partial_x(\rho v)&=0,\label{eq:continuity}\\
\partial_t(\rho v)+\partial_x(\rho v^2)&=-(W'*\rho)\,\rho.\label{eq:momentum}
\end{empheq}
\end{subequations}
Here \([0,T]\ni t\rightarrow \rho_t\) is a time-dependent family of Borel probability measures on \(\R\), and \([0,T]\ni t\rightarrow v_t\) is a velocity field defined \(\rho_t\)-almost everywhere. From the interaction  potential $W:\mathbb{R}\rightarrow\mathbb{R}$ we choose $W'(x)$ as a particular value in the sub or superdifferential set at $x\in\mathbb{R}$.  If the initial probability distribution is a convex combination of point masses, this system governs the dynamics of a collection of particles in which each particle interacts pairwise according to the potential $W$. If there is no interaction ($W=0$), the system is known as the pressureless Euler system. For $W(x)=\pm |x|$ we are looking at the attractive and repulsive Euler-Poisson systems, respectively. We are looking for a special type of solutions, called Sticky Particles solutions (SPS); what this means is that when particles collide, they undergo perfectly inelastic collisions so they stick together to form a new particle whose mass equals the sum of all the particles involved in the collision. Conservation of momentum through collisions is also imposed, i.e. the newly formed particle starts off at the time of collision with a momentum equal to the sum of the momenta of all the participating particles right before the collision. This model was first introduced by Zeldovich \cite{Zeldovich} as a cosmological model to study the formation of large structures in the universe. It is described as follows:  if $m_i$, $i=1,...,n$ is a discrete system of masses initially located at $-\infty < y^0_1 < ... < y^0_n < +\infty$ and moving with initial velocities $v^0_i$, $i=1,...,n$, then one makes the assumption that the acceleration of the $i^{th}$ particle is given by
$$y_i''(t)=-\sum_{j=1}^nm_jW'(y_i(t)-y_j(t))$$ while this particle undergoes no collision.  At the collision of a group of particles, the particles stick together and the initial velocity of the newly formed particle is given by the conservation of momentum. The evolution of the empirical distribution and its associated velocity
$$\rho_t:=\sum_{i=1}^nm_i\delta_{y_i(t)},\ v_t(y_i(t))=y_i'(t)$$ is described by \eqref{eq:continuity} and \eqref{eq:momentum}. In most of the works on existence of solutions the initial distribution is approximated by convex combinations of Dirac masses and the ensuing Sticky Particles system is used to approximate solutions to the system. 
Pressureless Euler (PE) has been studied by different techniques in \cite{Brenier1998}, \cite{Hynd2019}, \cite{Natile_2009}, \cite{Nguyen2008}, \cite{Tudorascu2015}, \cite{weinan1996}.  These techniques include the Sticky Particles model, description of the problem by an alternative scalar conservation law problem, and a semigroup approach. We would also like to acknowledge the contribution in \cite{Natile_2009}, where the solution to the pressureless Euler system is constructed by a Sticky Particles semigroup approach and a similar Lagrangian representation of the solution is obtained, albeit in a slightly weaker sense.

In \cite{SuderTudor} it was shown that Lagrangian solutions for pressureless Euler as in \cite{Hynd2019} satisfy the strong initial continuity of the energy, which, together with the Oleinik condition, is a necessary ingredient for the uniqueness result. The latter was obtained by means of explicitly linking Hynd's Lagrangian solutions \cite{Hynd2019} and the Scalar Conservation Laws solutions (SCL solutions; see \cite{Brenier1998} and \cite{Nguyen2008}) to the distributional solutions to a related problem whose well-posedness was shown in \cite{HuangWang}. More precisely, in \cite{SuderTudor} it is shown that if $v_0$ is right-continuous and bounded on the real line, then any Lagrangian solution must coincide with the (unique) SCL solution. Besides the strong initial continuity of the energy condition (SICE), a stronger version of the Oleinik condition was necessary for this connection (with the solutions in \cite{HuangWang}) to work: the velocity $v$ admits a Borel representative $\hat v$ (i.e. $\hat v(t,\cdot)=v(t,\cdot)$, $\rho(t,\cdot)$-a.e. for all $t$) which satisfies a stronger version of the Oleinik condition, namely {\it everywhere} Oleinik, in the sense that the pertinent inequality is required to hold {\it everywhere} (and not just $\rho(t,\cdot)$-a.e.). This class of solutions is stable and contains the discrete Sticky Particles solutions \cite{SuderTudor}, so we call such solutions SPS (Sticky Particles Solutions). Outside the pressureless Euler case (i.e. $W\equiv 0$), few results on uniqueness of SP solutions exist in the literature. 
 
As far as the asymptotic behavior of solutions is concerned, in \cite{TudorSPS}, the author employed SPS and a reflection principle to solve the pressureless Euler problem in the case where $\mathcal{C}$ is an arbitrary closed subset of the real line. Of course, boundary conditions need be imposed when $\partial\mathcal{C}$ is nonempty (that is, $\mathcal{C}\neq\R$) and in \cite{TudorSPS} it is successfully argued that among the two natural choices, i.e. reflective vs sticky boundary, only the latter leads to a well-posed problem. Soon after, in \cite{hynd2023asymptoticsstickyparticlesevolution}, it was shown that if $\mathcal{C}$ is compact, then the SPS to PE converges asymptotically to an equilibrium solution. Tudorascu and Wassink \cite{tudorascu-wassink} showed that there is no universal rate associated with this asymptotic decay, i.e. for any positive convex function $\varphi$ on $[0,\infty)$ which converges asymptotically to zero, there exists an initial configuration for which the Wasserstein distance between $\rho_t$ and the equilibrium $\rho_\infty$ exceeds  
$\varphi$.
In \cite{Tudorascu2024REU} the authors studied the asymptotic behavior of SPS to the Cauchy problem associated with the pressureless Euler-Poisson (PEP) with attractive interaction between particles, namely $W(x)=|x|$. It turns out that the solution $\rho$ converges to some limiting probability distribution $\rho_\infty$ as $t\rightarrow\infty$ under far fewer restrictions than in the PE case. At the heart of the approach lies the Lagrangian representation of solutions in terms of the projection onto the convex cone of essentially nondecreasing functions in the Hilbert space $L^2(0,1)$, denoted by $\mathcal{K}$. More specifically, for the PE system, Natile and Savar\'{e} \cite{Natile_2009} proved that the optimal map $N_t$ that pushes the Lebesgue measure restricted to the interval $(0,1)$ forward to $\rho_t$ is given by 
$$N_t=\mathrm{proj}_\mathcal{K}\big(N_0+tV_0)\mbox{ for all }t\geq0,$$
where $V_0:=v_0\circ N_0$. The initial conditions are to be understood in the sense of distributions:
\begin{equation}\label{initcondexp}
\lim_{t\rightarrow 0^+}\int_\R\zeta(y)\rho(t,dy)=\int_\R\zeta(y)\rho_0(dy),\ \ \lim_{t\rightarrow 0^+}\int_\R v(t,y)\zeta(y)\rho(t,dy)=\int_\R v_0(y)\zeta(y)\rho_0(dy)
\end{equation}
for all $\zeta\in C_c^\infty(\R)$. 
In \cite{Tudorascu2024REU} the above projection formula was extended to the PEP case, for which it turns out we have:
$$N_t=\mathrm{proj}_\mathcal{K}\bigg(N_0+tV_0+\frac{1}{2}t^2A_0\bigg)\mbox{ for all }t\geq0,$$
where $A_0:=1-2\,\mathrm{Id}$ on $(0,1)$. This is
the term that accounts for the attractive interaction between particles in the PEP case. In \cite{TudorascuREU2025} the authors provide sufficient conditions for the solution to decay to equilibrium in the repulsive case, i.e. $W(x)=-|x|$. Here the SPS solution does not enjoy an explicit projection representation as in the attractive case and, in fact, the SPS do not make up a stable class of solutions even in the discrete case. That is, it is easy to construct a sequence of discrete SPS that converges to a discrete distributional solution that is not sticky.  In a companion to this paper we analyze stationary solutions and the asymptotic behavior of the above system with the repulsive-attractive interaction potential
\begin{equation}
W(x)\coloneqq \lvert \lvert x\rvert-1\rvert\label{eq:unshifted-potential}
\end{equation}
and some approximations of it.  
For such a potential we define the derivative $W'$ to be zero at the points where $W$ is not differentiable, i.e. $W'(-1)=W'(0)=W'(1)=0$. 

The main challenge in the case of the potential \eqref{eq:unshifted-potential}, unresolved at this point, is proving existence of solutions.  In \cite{Hynd2019}, \cite{Hynd2020} we find the most general (in regards to $W$) result available in the literature on existence of solutions for the initial value problem associated with \eqref{eq:continuity}, \eqref{eq:momentum}.  
Since this potential is not $c$-convex for any $c\in\mathbb{R}$, the methods of \cite{Hynd2019}, \cite{Hynd2020} do not apply here to prove convergence for the discrete solutions starting at approximations of the initial data to the distributional solution originating at said data; the author makes the point that the $c$-convexity assumption prevents the interaction from being ``too repulsive". Existence of distributional solutions in the purely repulsive case was obtained in \cite{Nguyen2008} by the SCL method but the solutions obtained as limits of discrete SPS are not necessarily sticky. In fact, Hynd \cite{Hynd2020} works with $c$-convex potentials that are also $C^1$, with the exception of $W(x)=|x|$ (attractive Euler-Poisson) which he handles by a limiting procedure. As we show in the companion paper, the singularities of \eqref{eq:unshifted-potential} at $\pm1$ can be handled in a similar manner. The $c$-convexity condition is violated by the concave singularity at $0$. To this end, we propose a different approach to obtaining general existence of solutions as limits of discrete SPS; this new approach requires a $C^1$ potential, so it does not apply to \eqref{eq:unshifted-potential}. However, it does not require the potential to be $c$-convex; the tradeoff is that the solutions we obtained are not necessarily sticky. In the next section we show that, under certain growth conditions on $W'$, discrete solutions starting in a compact interval of diameter $2M$ and with speeds bounded by $M$ stay in a compact interval for all times and the diameter of said interval depends only on $M$. That leads to a general existence result for distributional solutions that start at compactly supported probability measures with continuous velocities.  

\subsection{Notation and definitions}
Throughout this manuscript we use $\mathcal{P}(\mathbb{R})$ to denote the set of all Borel probabilities on the real line. More particularly, we use $\mathcal{P}_p(\mathbb{R})$ for all the elements of $\mathcal{P}(\mathbb{R})$ with finite $p$-moment, where $1\leq p<\infty$. The $p$-Wasserstein distance between two measures $\mu,\ \nu\in \mathcal{P}_p(\mathbb{R})$ is defined as
$$W_p(\mu,\nu):=\inf_{\gamma\in \Gamma(\mu,\nu)}\Big[\int\!\!\int|x-y|^p\gamma(dx,dy)\Big]^{1/p},$$
where $\Gamma(\mu,\nu)$ is the set of all joint Borel probability measures with marginals $\mu$ and $\nu$, respectively, i.e. 
$$\int\!\!\int\zeta(x)\gamma(dx,dy)=\int\zeta(x)\mu(dx)\mbox{ and }\int\!\!\int\zeta(y)\gamma(dx,dy)=\int\zeta(y)\nu(dy)$$
for any $\zeta\in C_b(\mathbb{R})$.
Since we work with Borel probabilities on the real line, we also have that 
$$W_p(\mu,\nu)=\|N^\mu-N^\nu\|_{L^p(0,1)},$$
where $N^\mu$ is the right-continuous optimal map pushing the Lebesgue measure restricted to the unit interval $(0,1)$, denoted by $\chi$, to $\mu$. It can be obtained as the generalized inverse of the right-continuous cumulative distribution function $M^\mu(y):=\mu((-\infty,y])$, i.e. $N^\mu(x)=\inf
\{y\in\mathbb{R}\,:\,M^\mu(y)>x\}$ for all $x\in (0,1)$.

The solution $(\rho,v)$ is an absolutely continuous curve in the Wasserstein space $\mathcal{P}_2(\mathbb{R})$ (see \cite{AGSBook}), in the sense that $v_t\in L^2(\rho_t)$ for a.e. $t\in[0,T]$, and the function $\beta(t):=\|v_t\|_{L^2(\rho_t)}$ belongs to $L^2(0,T)$. Furthermore, we have 
$$W_2(\rho_s,\rho_t)\leq \int_s^t\beta(\tau)d\tau\mbox{ for all }0\leq s\leq t.$$

The total energy (Hamiltonian) of a solution $(\rho,v)$ at time $t$ is
\begin{equation}\label{eq:hamiltonian}
\mathcal{H}(t):=\frac{1}{2}\int_{\mathbb{R}}v_t(x)^2\,\rho_t(dx)+\frac{1}{2}\iint_{\mathbb{R}\times\mathbb{R}}W(x-y)\,\rho_t(dx)\,\rho_t(dy),
\end{equation}
the sum of the kinetic energy and the interaction energy. Along sticky-particle solutions $\mathcal{H}$ is non-increasing since sticky collisions dissipate kinetic energy; see \eqref{eq:nonincresing-energy}.

\subsection{Organization of the paper}
This paper is organized as follows: Section 2 identifies classes of potentials with the confinement property, i.e. for any $M>0$ there exists $C(M)>0$ such that the SPS originating at an initial discrete probability measure supported in $[-M,M]$ and starting off with a speed of at most $M$ remains confined to the interval $[-C(M),C(M)]$ along with its velocity for all time. In Section 3 we collect some results on the Lagrangian representation of discrete SPS $(\rho,v)$ and rewrite them in terms of the optimal maps pushing the uniform probability measure on the unit interval forward to $\rho_t$. This will allow us to pass to the limit in Section 4 and obtain existence of general distributional solutions for compactly supported initial distributions with continuous (on the support) velocities. Since the potentials analyzed in this paper are attractive at long range, we dedicate Section 5 to comparing the behavior of our SPS to the SPS for the attractive case $W(x)=|x|$. The latter case was shown earlier \cite{Tudorascu2024REU} to decay at an explicit rate to a single Dirac delta; in fact, the collapse into the equilibrium happens in finite time for solutions originating in a compact set. In the Appendix, we apply the confinement results shown in Section 2 to the case of a semiconvex potential (\cite{Hynd2020}), where we show time-shifted trajectories of SPS converge to constant-energy eternal solutions of a $L^2$-valued ODE, which allows for a short proof of an asymptotic collapse result.  

    \section{Confinement for discrete SP solutions} \label{sec:uniformconfinement}

    In this section, we establish uniform-in-time confinement estimates for discrete sticky-particle solutions under interaction potentials with controlled growth at infinity. We combine moment estimates and estimate a certain quantity that is coercive on support diameter to show confinement. The results provide the compactness needed in the subsequent sections to pass from discrete sticky-particle approximations to distributional solutions for general compactly supported initial data.

        \textit{We require once and for all that $W \geq 0$. Any admissible potential can be made to satisfy this by adding a finite constant. This does not affect $W'$, thus the pressureless Euler equation remains unchanged.}

    \begin{definition} \label{defn:admissiblepotential}
        Let $W \in W^{1, \infty}_{\mathrm{loc}}(\mathbb{R})$ be an even function. For any $p \in [0, 1)$, we say $W$ is an \emph{admissible potential} for exponent $p$ if
        $$0 < \liminf_{x \rightarrow \infty} \frac{W'(x)}{|x|^p} \leq \limsup_{x \rightarrow \infty} \frac{W'(x)}{|x|^p} < \infty,$$
        and for the endpoint case of $p=1$ we only require that
        \[\liminf_{x \rightarrow \infty} \frac{W'(x)}{|x|} > 0.\]
        Note that the endpoint case includes all such potentials where $W'$ grows super-linearly, and hence we do not need to consider $p > 1$.
    \end{definition}
    \begin{remark} \label{rem:boundedforceconfinement}
        For $p=0$ (the case of \emph{bounded-force potentials}, i.e. $\norm{W'} < \infty$), uniform confinement (\Cref{thm:supportbound}) holds even when the short-range repulsion dominates the long-range attraction, i.e. when $$\norm{\max\{W',0\}}_{L^\infty(0,\infty)} < \norm{\min\{W',0\}}_{L^\infty(0,\infty)}.$$
\end{remark}
    \begin{remark}        
        For ease of exposition, the proofs in the subsequent subsections will nevertheless assume both an upper and lower bound on the growth of $W'$: that is, they apply for all admissible potentials with exponent $p \in [0, 1)$ as defined, and additionally for $p=1$ if an upper bound on growth is assumed as well. For full generality (i.e. removing the upper bound on growth for $p=1$), an analogous but more delicate argument is needed, which is developed in \Cref{subsect:comtracking}.

        Indeed, the argument for confinement of discrete SP solutions can be proven under the slightly weaker assumption that $x^p \lesssim W'(x) \lesssim x^q$ for $p, q \geq 0$ and $q < p + 1$ (where the upper bound is again superfluous for $p=1$). For notational clarity we will not pursue this.
    \end{remark}        
    Note that although functions in $W^{1, \infty}_{\mathrm{loc}}(\mathbb{R})$ only have a.e.-defined weak derivative, we require that the growth condition on $W'$ hold \emph{everywhere}: that is, we fix an odd  pointwise representative of $W'$. This is consistent with the sticky-particle formulation of one-dimensional pressureless Euler, as defining the convolution $(W' * \rho)$ for arbitrary $\rho$ requires fixing a representative of $W'$ as well.
    
    Also, note that $W^{1, \infty}_{\mathrm{loc}}(\mathbb{R}) = C^{0, 1}_{\mathrm{loc}}(\mathbb{R})$, which will be of use later.  
    
    The goal is to show the following:
    \begin{restatable}[Support Diameter Bound]{theorem}{restatesupportbound} \label{thm:supportbound}
        Let $W$ be an admissible potential. For any finite weighted sum of Dirac deltas $\rho_0$ and $v_0\in C(\mathbb{R})$ such that
        \begin{equation}\nonumber
        \int y \rho_0(dy) = \int v_0(y) \rho_0(dy) = 0, 
        \end{equation}
        let $(\rho, v)$ denote the SPS corresponding to the initial data $(\rho_0, v_0)$. 
        Then for any $0<M<\infty$ there exists $0<C(M)<\infty$ such that for any initial $(\rho_0, v_0)$ satisfying the previous hypothesis and additionally $\rho_0$ supported in $[-M,M]$ and any $v_0\in C(\mathbb{R})$ with $|v_0|\leq M$ we have $\mathrm{spt} (\rho_t) \subset [-C(M), C(M)]$ for all time $t$.
    \end{restatable}
    As we will see, this immediately yields the following:
    \begin{restatable}[Velocity Bound]{corollary}{restatevelocitybound} \label{corr:velocitybound}
        Under the same assumptions as \Cref{thm:supportbound}, there exists some $C'(M) > 0$ such that $\norm{v_t}_{L^\infty(\rho_t)} \leq C'$ for all $t$.
    \end{restatable}

    \subsection{Preliminary lemmas} \label{subsect:lemmas}
    \begin{lemma} \label{lem:momentbounds}
        Under the same assumptions as \Cref{thm:supportbound}, for all $0 \leq q \leq p+1$ there exists a $\Delta_{q}(M) > 0$ such that $$\int |y|^{q} \rho_t(dy) \leq \Delta_{q} \qquad \forall t,$$
        where $p$ is the growth exponent as in \Cref{defn:admissiblepotential}.
    \end{lemma}
    \begin{proof}
        By the growth condition and evenness of $W$, there exists some $R > 0$ and $\delta > 0$ such that $$\frac{W(x)}{1 + |x|^{p+1}} \geq \delta > 0 \qquad \forall |x| \geq R.$$
        Additionally, as $W \in W^{1, \infty}_{\mathrm{loc}}(\mathbb{R})$, there exists some constant $s$ (depending only on $W$) such that $$W(x) \geq \delta |x|^{p+1} + s\mbox{ for all }x\in\mathbb{R}.$$
        So by Jensen's inequality, we have that 
        \begin{align*}
            \iint W(x - y) \rho_t(dx) \rho_t(dy) &= \iint W(|x - y|) \rho_t(dx) \rho_t(dy)\\
            &\geq \delta \iint |x - y|^{p+1} \rho_t(dy) \rho_t(dx) + s\\
            &\geq \delta \int \left|x - \left( \int y \rho_t(dy) \right) \right|^{p+1} \rho_t(dx) + s\\
            &= \delta \int |x|^{p+1} \rho_t(dx) + s,
        \end{align*}
        where the last equality follows because we assumed that $\int y \rho_0(dy) = 0$ and $\int v_0(y) \rho_0(dy) = 0$, which combined with the evenness of the potential implies that the center of mass is zero for all time. Thus, 
        $$\int |x|^{p+1} \rho_t(dx) \leq \frac{1}{\delta} \left( -s + \iint W(x - y) \rho_t(dx) \rho_t(dy) \right).$$
        And now, because the Hamiltonian $\mathcal{H}(t)$ is non-increasing, $$\iint W(x - y) \rho_t(dx) \rho_t(dy) \leq 2 \mathcal{H}(0),$$ which has an upper bound depending on $M$ only (as long as $\rho_0$ is supported in $[-M,M]$ and $|v_0|\leq M$ on $[-M,M]$). Thus, there exists some $\Delta_{p+1}(M) > 0$ such that $\int |x|^{p+1} \rho_t(dx) \leq \Delta_{p+1}$ for all $t$.
        
        The fact that for all $1\leq q\leq p+1$ a suitable $\Delta_q$ exists follows by the existence of $\Delta_{p+1}$ and the standard truncation argument. Extending these bounds to cover the full range of $0 \leq q \leq p+1$ follows immediately by Jensen's inequality.
    \end{proof}
    We now make the following definition:
    \begin{definition} \label{defn:forceatapoint}
        Let $(\rho_t, v_t)$ be as in \Cref{thm:supportbound}. For any $t$, define the $\rho_t$-measurable function $F_t$ by $$F_t(x) := -\int_{\mathbb{R}} W'(x - y) \rho_t(dy).$$
    \end{definition}
    With the moment bound, we have the following estimates on $F_t$:
    \begin{lemma} \label{lem:globalforcebounds}
        Let $(\rho_t, v_t)$ be as in \Cref{thm:supportbound}. Then for any $R > 0$, there exists a $\Lambda_R > 0$ such that $$x \geq -R \implies F_t(x) \leq \Lambda_R.$$
    \end{lemma}
    \begin{proof}
        \thmendspace
        By the growth bound, there exists some $R_1 > 0$ such that $W'(x) > 0$ for all $x \geq R_1$. And as $W \in W^{1, \infty}_{\mathrm{loc}}(\mathbb{R})$, define $$Q := \norm{W'}_{L^\infty([-R_1, R_1])}.$$
        For any $x$ we can bound
        \begin{align*}
            F_t(x) &= -\int_{(-\infty, x]} W'(x - y) \rho_t(dy) + \int_{(x, \infty)} W'(y - x) \rho_t(dy)\\
            &\leq -\int_{(x - R_1, x]} W'(x - y) \rho_t(dy) + \int_{(x, \infty)} W'(y - x) \rho_t(dy)\\
            &\leq Q + \int_{(x, \infty)} W'(y - x) \rho_t(dy).
        \end{align*}
        Thus, this lemma is reduced to showing that the second term in the final inequality is uniformly bounded (depending only on $M, R$) whenever $x \geq -R$. To that end, note that the upper growth bound $$\limsup_{x \rightarrow \infty} \frac{W'(x)}{|x|^p} < \infty$$ combined with the local boundedness of $W'$ implies that there exists some $\delta > 0, K \in \mathbb{R}$ such that $$W'(x) \leq \delta |x|^p + K \qquad \forall x \geq 0.$$
        Now, letting $\Delta_{p}$ be as in \Cref{lem:momentbounds} we can write:
        \begin{align*}
            \int_{(x, \infty)} W'(y - x) \rho_t(dy) &\leq \delta \int_{(x, \infty)} (|y| + R)^p \rho_t(dy) + K\\
            &\leq \delta \int_{(x, \infty)} (|y|^p + |R|^p) \rho_t(dy) + K\\
            &\leq \delta \Delta_p + \delta |R|^p + K.
        \end{align*}
        Thus, setting $\Lambda_R := Q + \delta (\Delta_p + |R|^p) + K$ concludes the lemma.
    \end{proof}
    To bound the support diameter, we derive a sharper bound at the tail:
    \begin{lemma} \label{lem:tailforcebounds}
        Let $(\rho_t, v_t)$ be as in \Cref{thm:supportbound}. Then there exists some $R > 0$ and $\Lambda < 0$ depending only on $M$ such that $$x \geq R \implies F_t(x) \leq \Lambda < 0.$$ 
    \end{lemma}
    \begin{proof}        
        \thmendspace
        Let $\Delta_1, \Delta_{p+1}$ be as in \Cref{lem:momentbounds}. By Markov's inequality, for any $r > 0$ we have $$\rho_t(\{|y| \geq r\}) \leq \frac{\Delta_1}{r}.$$
        By the lower growth bound, there exists some $R' > 0$, $\delta > 0$ such that $W'(x) \geq \delta$ for all $x \geq R'$. Similar to the previous lemma, by the upper growth bound there exists some $Q > 0$ such that $$W'(x) \leq Q (|x|^p + 1) \qquad \forall x \geq 0.$$
        Now, for any $x, s$ satisfying $x \geq s > R'$, we have:
        \begin{align*}
            F_t(x) &= \int_{\mathbb{R}} W'(y - x) \rho_t(dy)\\
            &\leq -\delta \rho_t((-\infty, x - R')) + \norm{W'}_{L^\infty([-R', R'])} \rho_t([x-R', x]) + \int_{(x, \infty)} W'(y - x) \rho_t(dy).
        \end{align*}
        By Markov's inequality, we can upper bound the first term by $$-\delta \rho_t((-\infty, x - R')) \leq -\delta (1 - \rho_t([s - R', \infty))) \leq - \delta \left(1 - \frac{\Delta_1}{s - R'}\right),$$
        and the second term by $$\norm{W'}_{L^\infty([-R', R'])} \rho_t([x - R', x]) \leq \norm{W'}_{L^\infty([-R', R'])} \frac{\Delta_1}{s - R'}.$$
        To estimate the final term, we use H\"{o}lder's inequality:
        \begin{align*}
            \int_{(x, \infty)} W'(y - x) \rho_t(dy) &\leq \int_{(x, \infty)} Q((y - x)^p + 1) \rho_t(dy)\\
            &\leq Q \rho_t((x, \infty)) + Q \int_{(x, \infty)} |y-x|^p \rho_t(dy)\\
            &\leq Q \cdot \frac{\Delta_1}{s} + Q \left( \int_{(x, \infty)} |y - x|^{p+1} \rho_t(dy) \right)^{p/(1+p)} \left( \int_{(x, \infty)} 1^{1+p} \rho_t(dy) \right)^{1/(1+p)}\\
            &\leq Q \cdot \frac{\Delta_1}{s} + Q \left[\rho_t((x, \infty))\right]^{1 / (1+p)} \left( \int_{(x, \infty)} |y|^{p+1} \rho_t(dy) \right)^{p/(1+p)}\\
            &\leq Q \cdot \frac{\Delta_1}{s} + Q \left(\frac{\Delta_1}{s}\right)^{1 / (1+p)} \Delta_{p+1}^{p/(1+p)}.
        \end{align*}
        The fact that the H\"{o}lder exponent $\frac{p+1}{p}$ is undefined for $p=0$ is not a problem: in that case the final term reduces to just $2Q \cdot \rho_t((x, \infty))$, which can be bounded in terms of $s$ by Markov's inequality. 
        
        Putting this all together, we have that $$F_t(x) \leq -\delta + E(s),$$
        where $E(s)$ is a quantity depending only on $s$ (and data in the hypothesis of \Cref{thm:supportbound}), satisfying $E(s) \rightarrow 0$ as $s \rightarrow \infty$. Thus, it is clear that we can take $s$ large enough such that $F_t(x) \leq -\frac{\delta}{2}$ for all $x \geq s$.

        Thus, setting $R := s$ and $\Lambda := -\frac{\delta}{2}$ completes the lemma.
    \end{proof}

    \subsection{Proof of uniform confinement} \label{subsect:uniformconfinementproof}
    We are now prepared to prove uniform confinement of our SPS systems. The main obstacle is that although \Cref{lem:tailforcebounds} provides a restoring force for particles that are ``far'' from the origin, repeated collisions between particles may transfer additional momentum to the rightmost particle, preventing a uniform diameter bound from being trivially obtained by kinematics of the rightmost particle.
    \restatesupportbound*
    \begin{proof}
        Let $R, \Lambda$ be as in \Cref{lem:tailforcebounds}. Note that it will be convenient to take $R > \max{(1, M, 2 \Delta_1)}$: this can always be done, because the conclusion of \Cref{lem:tailforcebounds} still holds after increasing $R$. Define the function $$\Phi(x) = \begin{cases} 0 & x \leq R \\ -\Lambda(x - R) & x > R \end{cases}$$
        Note that as $\Lambda < 0$, this implies that $\Phi$ is monotone increasing (strictly so for $x \geq R$), and $\lim_{x \rightarrow \infty} \Phi(x) = +\infty$. 

        Now, suppose at any time $t$ the system $(\rho_t, v_t)$ has particles with position, velocity, mass (resp.) given by $$\{x_1(t), \ldots, x_{N(t)}(t)\}, \qquad \{m_1, \ldots, m_{N(t)}\}, \qquad \{v_1(t), \ldots, v_{N(t)}(t)\},$$
        where without loss of generality we can take $x_1(t) < \ldots < x_{N(t)}(t)$. 

        Define the position-adjusted energy to be
        \begin{equation} \label{eqn:modifiedenergy}
            E_j(t) := \frac{1}{2} (v_j(t))_+^2 + \Phi(x_j(t)), \qquad E(t) := \max_j [E_j(t) \cdot \mathds{1}_{[R, \infty)}(x_j(t))].
        \end{equation}
        We now claim that $E(t)$ is uniformly bounded above for all $t$, i.e. there exists some $C_1(M) > 0$ such that $E(t) \leq C_1$ for all $t$. Consider the following cases:
        \begin{enumerate}
            \item $t=0$. $\rho_0$ is supported in $[-M, M]$, and we took $R$ large enough such that $R > M$. So $E_j(0) = \frac{1}{2} (v_j(0))_+^2$ for all $j$, and as $\norm{v_0}_{L^\infty(\rho_0)} \leq M$, this means that $E(0) \leq \frac{1}{2} M^2$. So $E(0)$ is uniformly bounded.
            \item At collisions. Suppose some particles collide. Now, observe that:
            \begin{itemize}
                \item $\Phi$ is identical for all particles involved in the collision, because positions must coincide for a collision to happen. So the resulting sticky particle has the same ``position energy'' $\Phi$ as the incoming particles. 
                \item The velocity of the sticky particle resulting from the collision is a convex combination of the velocities of the incoming particles, and the map $v \mapsto \frac{1}{2} v_+^2$ is convex. Thus, the ``rightwards kinetic energy'' component of the resulting sticky particle is bounded above by the largest ``rightwards kinetic energy'' of any incoming particle.
            \end{itemize}
            Thus, $E(t)$ does not increase at collisions.
            \item Between collisions. As the max in the definition of $E$ only includes particles in the range $[R, \infty)$ at time $t$, the only two ways $E$ can increase are the following:
            \begin{itemize}
                \item Particle $x_i$ enters the region $[R, \infty)$ at time $t$. Then by \Cref{lem:tailvelocitybound}, there exists some universal $\overline{C}(M) > 0$ such that $$(v_i(t))_+ \leq \overline{C}.$$
                And so $$E_i(t) \leq \frac{1}{2} \overline{C}^2 + \Phi(x_i(t)) = \frac{1}{2} \overline{C}^2 + \Phi(R) = \frac{1}{2} \overline{C}^2.$$
                \item Particle $x_i$ satisfies $x_i(t) > R$. We will show that $\dot{E}_i(t) \leq 0$. If $v_i(t) \leq 0$ the situation is trivial: $(v_i(t))_+^2 = 0$ is constant, and $\Phi(x_i(t))$ is decreasing by monotonicity of $\Phi$. So $E_i(t)$ is decreasing at $t$. If $v_i(t) > 0$, we can directly compute $$\dot{E}_i(t) = v_i(t)(\dot{v}_i(t) + \Phi'(x_i(t))) \leq 0,$$
                where the inequality holds because $v_i(t) > 0$ and $\dot{v}_i(t) + \Phi'(x_i(t))$ is non-positive by \Cref{lem:tailforcebounds}.
            \end{itemize}
        \end{enumerate}
        Thus, $E_i(t) \leq \max{(\frac{1}{2} M^2, \frac{1}{2} \overline{C}^2)}$ for all $i$ and all $t$ such that $x_i(t) \geq R$. Therefore, $E(t)$ is uniformly bounded for all time $t$, which implies that $\Phi(x_i(t))$ is uniformly bounded for all $i$ and all $t$.
        
        By coercivity of $\Phi$, this implies that there is some $C(M) > 0$ such that $x_i(t) \leq C$ for all $i$ and $t$, which implies that $\sup (\mathrm{spt} (\rho_t)) \leq C$ for all $t$. 
        
        Since equations \eqref{eq:continuity} and \eqref{eq:momentum} are invariant under the reflection
        $$(\rho_t,v_t)\mapsto\bigl((-\mathrm{id})_\#\rho_t,\,-v_t(-\cdot)\bigr),$$
        the same uniform bound holds for \(\inf(\operatorname{spt}\rho_t)\), completing the proof.
    \end{proof}
    And now we prove the lemma, which completes the proof of \Cref{thm:supportbound}:
    \begin{lemma} \label{lem:tailvelocitybound}
        Let $(\rho_t, v_t), W, M$ be as in \Cref{thm:supportbound}, and let $R, \Lambda$ be as in \Cref{lem:tailforcebounds}. Without loss of generality, take $R > \max{(1, 2 \Delta_1)}$, where $\Delta_1$ is from \Cref{lem:momentbounds} (if we increase $R$, \Cref{lem:tailforcebounds} still holds).
        
        Then there exists some $\tilde{C}(M) > 0$ such that for all time $t$, the following holds: $$v_t(x) \leq \tilde{C} \qquad \forall x \in \mathrm{spt} (\rho_t) \cap [R, \infty).$$
    \end{lemma}
    \begin{proof}
        \thmendspace
        Fix a time $T$, and suppose at time $T$ the system $(\rho_T, v_T)$ has particles with position, velocity, mass (resp.) given by $$\{x_1(T), \ldots, x_{N(T)}(T)\}, \qquad \{m_1, \ldots, m_{N(T)}\}, \qquad \{v_1(T), \ldots, v_{N(T)}(T)\}.$$
        Choose a $j$ that maximizes $(v_j(T))_+$ over all $j$ that satisfy $x_j(T) \geq R$ (if there is no such $j$ then this lemma is vacuously true). In our system, particles collide stickily. Follow the path of collisions for $v_j$ back in time, and at each collision, track the particle involved in the collision that has the \emph{largest} velocity $v(t)$. Call this trajectory $X(t)$, and let the velocity be $V(t) := \dot{X}(t)$ when it exists (note that $V(t)$ exists except at finitely many $t$).
        
        So $X(t)$ is a continuous curve, and $V(t)$ is a piecewise continuous curve with finitely many jump discontinuities $\{t_1, \ldots, t_k\}$, satisfying $V(t_i+) \leq V(t_i-)$ for all $i$.

        Consider the following cases on $\inf_{t \leq T} X(t)$:
        \begin{enumerate}
            \item $\inf_{t \leq T} X(t) \leq -R$. Because $X(t)$ is continuous, let $t^*$ be time where this $\inf$ is attained. Because our center of mass is always $0$, and our first moment is bounded above by some $\Delta_1(M) < \infty$ (\Cref{lem:momentbounds}), Markov's inequality gives us $$\rho_{t^*}((-\infty, -R]) \leq \frac{\Delta_1}{R} \implies \rho_{t^*}((X(t^*), \infty)) \geq 1 - \frac{\Delta_1}{R}.$$
            At time $T > t^*$, we know that $X(T) \geq R$. By Markov's inequality again,
            $$\rho_T((X(T), \infty)) \leq \rho_T((R, \infty)) \leq \frac{\Delta_1}{R}.$$
            Since trajectories for particles cannot cross, we know that $$m_j \geq \rho_{t^*}([X(t^*), \infty)) - \rho_T((X(T), \infty)) \geq 1 - 2 \frac{\Delta_1}{R} > 0.$$
            Now, notice this \emph{lower bound} on $m_j$ is actually uniform. We consequently use a trivial energy bound: $$\frac{1}{2} \left(v_j(T)\right)^2 m_j \leq \mathcal{H}(T) \implies (v_j(T))_+ \leq \sqrt{2 \cdot \left(1 - \frac{2 \Delta_1}{R}\right)^{-1} \cdot \mathcal{H}(T)}.$$
            But, of course, the Hamiltonian is non-increasing so $\mathcal{H}(T) \leq \mathcal{H}(0)$, and $\mathcal{H}(0)$ is bounded above in terms of just $M$. And so $(v_j(T))_+$ is also bounded above by some $C'(M) > 0$ in this case.
            \item $\inf_{t \leq T} X(t) \geq -R$. Then let $\Lambda_R$ be as in \Cref{lem:globalforcebounds}, and as such $X(t)$ solves the differential inequalities 
            \[\begin{cases}
                |X(0)|, |\dot{X}(0)| \leq M\\
                X(t) \geq -R \qquad \forall t\\
                \ddot{X}(t) \leq \Lambda_R \qquad \forall t\\
                \ddot{X}(t) \leq \Lambda < 0 \qquad \text{when } X(t) \geq R,
            \end{cases}\]
            except at finitely many discontinuities $\{t_1, \ldots, t_k\}$, where, by construction, $\dot{X}(t_\ell-) > \dot{X}(t_\ell+)$ for all $\ell \in \{1, \ldots, k\}$. Now, let $$s := \sup \{\tau \leq T : \tau = 0 \text{ or } \dot{X}(\tau) \leq 0\}.$$
            We know that $\dot{X}(t) > 0$ on $(s, T]$, i.e. $X(t)$ is strictly increasing on $t \in (s, T]$ and thus invertible. Furthermore, $\dot{X}(s) \leq M$ and so $$(\dot{X}(T))^2 \leq (\dot{X}(s))^2 + 2 \int_{X(s)}^{X(T)} \ddot{X}(X^{-1}(z)) dz \leq M^2 + 2 \int_{X(s)}^{X(T)} \ddot{X}(X^{-1}(z)) dz.$$
            Notice that whenever $X(t) \geq R$ we have $\ddot{X}(X^{-1}(X(t))) \leq \Lambda < 0$, and as $X(T) \geq R$, we can thus bound
            $$\int_{X(s)}^{X(T)} \ddot{X}(X^{-1}(z)) dz \leq \int_{X(s)}^{R} \ddot{X}(X^{-1}(z)) dz.$$
            And so $(\dot{X}(T))^2\leq M^2 + 4 \Lambda_R R$, which implies $(\dot{X}(T))_+ \leq \sqrt{M^2 + 4 \Lambda_R R}$.
        \end{enumerate}
        We have a uniform bound in each case, and as such we are done.
    \end{proof}
    In light of the proof of \Cref{lem:tailvelocitybound}, the proof of \Cref{corr:velocitybound} is fairly immediate, using the same backwards trajectory tracking trick:
    \restatevelocitybound*
    \begin{proof}
        Let $C(M) > 0$ be as in \Cref{thm:supportbound}, and let $F_t$ be as in \Cref{defn:forceatapoint}. Note that as a posteriori $\mathrm{spt} (\rho_t) \subset [-C, C]$ for all time $t$,  by the fact that $W \in W^{1, \infty}_{\mathrm{loc}}(\mathbb{R})$ we know that there exists some $L > 0$ such that $$|F_t(x)| \leq L \qquad \forall x \in \mathrm{spt} (\rho_t) \qquad \forall t.$$
        Now, we show that there exists some $C'(M) > 0$ such that $v_t(x) \leq C'$ for all $x \in \mathrm{spt} (\rho_t)$ and $t \geq 0$. A symmetric argument will then establish the velocity bound.

        Fix any time $T > 0$, and represent the system at time $T$ as having particles $\{x_1(T), \ldots, x_{N(T)}(T)\}$ with velocities $\{v_1(T), \ldots, v_{N(T)}(T)\}$. Pick an index $j^*$ that maximizes $v_j(T)$ over all $j$. Similar to \Cref{lem:tailvelocitybound}, track the trajectory of particle $j^*$ backwards in time, and at each collision selecting the highest-velocity ancestor.

        This yields a continuous curve $X: [0, T] \rightarrow [-C, C]$, such that $X$ is differentiable on $[0, T]$ except at finitely many times (the collisions): at each such time $s$, we have by construction that $\dot{X}(s-) \geq \dot{X}(s+)$. And moreover, between collisions we know that $\ddot{X}(t) \leq L$. Thus, by a work-energy theorem argument similar to that in \Cref{lem:tailvelocitybound}, we know that $$v_{j^*}(T) = \dot{X}(T) \leq \sqrt{M^2 + 4 C L},$$
        which is a uniform bound.
    \end{proof}
  
    \subsection{Sharper estimates for $p = 1$} \label{subsect:comtracking}
    Throughout this section we assume \emph{only} the following growth bound: $$\liminf_{x \rightarrow \infty} \frac{W'(x)}{|x|} > 0.$$
    We emphasize that the arguments in earlier sections still hold for $p=1$ if both a lower and upper bound on the growth of $W'$ are assumed; here we show that in the case of $p=1$ the upper bound on the growth of $W'$ is superfluous.

    Recall that in the proof of \Cref{thm:supportbound}, we obtained (in \Cref{subsect:lemmas} and \Cref{subsect:uniformconfinementproof}) bounds on velocities $v_i(t)$ and accelerations $\dot{v}_i(t)$ of particles when the position $x_i(t)$ is large. We then defined the position-adjusted energy \Cref{eqn:modifiedenergy}, and utilized earlier bounds to show that such a modified energy is bounded. In this section, we will develop an appropriate analogue to the position-adjusted energy \Cref{eqn:modifiedenergy} that remains uniformly bounded in the $p=1$ case even \emph{without} assuming an upper bound on the growth of $W'$.
    
    The following object will be important in the sequel:
    \begin{definition} \label{defn:rightcenterofmass}
        Let our system at time $t$ have particles with position, mass, velocity (resp.) given by $$\{x_1(t), \ldots, x_{N(t)}(t)\}, \qquad \{m_1, \ldots, m_{N(t)}\}, \qquad \{v_1(t), \ldots, v_{N(t)}(t)\},$$
        and without loss of generality assume that $x_1(t) < \ldots < x_{N(t)}(t)$.
        Define the quantities $c_j(t)$ and $M_j(t)$ as follows:
        $$c_j(t) := \fint_{[x_j(t), \infty)} x \, \rho_t(dx) \qquad M_j := \rho_t([x_j(t), \infty)).$$
        Roughly, this is the ``rightwards center of mass'' of $x_j$. Note that we do not notate $M_j$ to depend on time, as it is constant between sticky collisions, and at sticky collisions we re-index.
    \end{definition}
    
    The following lemma is analogous to \Cref{lem:tailforcebounds}.
    \begin{lemma} \label{lem:clustertailforcebounds}
        Let $(\rho_t, v_t)$ be as in \Cref{thm:supportbound}, and $$\{x_1(t), \ldots, x_{N(t)}(t)\}, \{m_1, \ldots, m_{N(t)}\}, \{v_1(t), \ldots, v_{N(t)}(t)\}, \{c_1(t), \ldots, c_{N(t)}(t)\}, \{M_1, \ldots, M_{N(t)}\}$$ be as in \Cref{defn:rightcenterofmass}. Then there exists some $R > 0$ and $\Lambda < 0$ such that for all $j$ satisfying $M_j \leq \frac{1}{2}$, we have $$c_j(t) \geq R \implies \ddot{c}_j(t) \leq \Lambda < 0.$$
    \end{lemma}
    \begin{proof}
        \thmendspace
        Let $\Delta_1 > 0$ be as in \Cref{lem:momentbounds}.
        
        We can write $$\ddot{c}_j(t) = -\fint_{[x_j(t), \infty)} \int_{(-\infty, x_j(t))} W'(x - y) \rho_t(dy) \rho_t(dx).$$
        By the $p=1$ growth bound, there exists some $R' > 1$ and $\delta > 0$ such that $W'(x) \geq \delta x$ for all $x \geq R'$. And as $W \in W^{1, \infty}_{\mathrm{loc}}(\mathbb{R})$, let $$Q := \norm{W'}_{L^\infty([-R', R'])},$$
        which also implies that $$W'(x) \geq \delta x - (\delta R' + Q) \qquad \forall x \geq 0.$$
        By Markov's inequality, for any $s > R'$ we have $\rho_t(\{|y| \geq s - R'\}) \leq \frac{\Delta_1}{s - R'}$. So we can choose $s$ large enough, such that $$\rho_t(\{|y| \geq s - R'\}) \leq \frac{\delta}{4(Q + \delta)}.$$
        Now, to bound $\ddot{c}_j(t)$, we can do casework on the leftmost point of the cluster $x_j(t)$:
        \begin{itemize}
            \item $x_j(t) \geq s$. Then we can bound 
            \begin{align*}
                \ddot{c}_j(t) \leq& -\fint_{[x_j(t), \infty)} \int_{(-\infty, x_j(t) - R']} W'(x - y) \rho_t(dy) \rho_t(dx)\\
                &- \fint_{[x_j(t), \infty)} \int_{(x_j(t) - R', x_j(t))} W'(x - y) \rho_t(dy) \rho_t(dx)\\
                \leq& \frac{\delta}{4(Q + \delta)} Q - \left(\frac{1}{2} - \frac{\delta}{4(Q + \delta)} \right) \delta = -\frac{\delta}{4}.
            \end{align*}
            And $x_j(t) \geq s$ implies $c_j(t) \geq s$, and so in this case the lemma holds for $R := s$ and $\Lambda = -\frac{\delta}{4}$.
            \item $x_j(t) \leq s$. As the repulsive component of $W'$ is bounded by $Q$, we can split and write:
            \begin{align*}
                \ddot{c}_j(t) \leq& -\frac{1}{M_j}\int_{[x_j(t), s)} \int_{(-\infty, x_j(t))} W'(x - y) \rho_t(dy) \rho_t(dx)\\
                &-\frac{1}{M_j}\int_{[s, \infty)} \int_{(-\infty, x_j(t))} W'(x - y) \rho_t(dy) \rho_t(dx).
            \end{align*}
            Write the first term as $A$, and second term as $B$, such that $\ddot{c}_j(t) \leq A + B$. Estimating $A$, we have the trivial bound
            $$A \leq \frac{1}{M_j} \int_{[x_j(t), s)} Q \rho_t(dx) \leq \frac{1}{M_j} Q \cdot \rho_t([x_j(t), s)) \leq Q.$$
            Estimating $B$, we get 
            \begin{align*}
                B &\leq -\frac{\delta}{M_j} \int_{[s, \infty)} \int_{(-\infty, x_j(t))} (x - y) \rho_t(dy) \rho_t(dx) + \frac{1}{M_j} (\delta R' + Q) \cdot \rho_t([s, \infty))\\
                &\leq -\frac{\delta}{M_j} \int_{[s, \infty)} \rho_t((-\infty, x_j(t))) \cdot (x - s) \rho_t(dx) + \frac{1}{M_j} (\delta R' + Q) \cdot M_j\\
                &\leq -\frac{\delta}{M_j} \int_{[s, \infty)} \frac{1}{2} \cdot (x - s) \rho_t(dx) + (\delta R' + Q)\\
                &\leq -\frac{\delta}{2} \frac{1}{M_j} \int_{[s, \infty)} x \rho_t(dx) + \frac{\delta}{2} \cdot s + (\delta R' + Q).
            \end{align*}
            We want $c_j(t)$ to be ``coercive'' on $B$. So we can re-write:
            \begin{align*}
                B &\leq -\frac{\delta}{2} c_j(t) + \frac{\delta}{2} \left( c_j(t) - \frac{1}{M_j} \int_{[s, \infty)} x \rho_t(dx) \right) + \frac{\delta}{2} s + (\delta R' + Q)\\
                &\leq -\frac{\delta}{2} c_j(t) + \frac{\delta}{2} \frac{1}{M_j} \left( \int_{[x_j(t), s)} x d\rho_t(x) \right) + \frac{\delta}{2} s + (\delta R' + Q)\\
                &\leq -\frac{\delta}{2} c_j(t) + \frac{\delta}{2} s + \frac{\delta}{2} s + (\delta R' + Q).
            \end{align*}
            And thus we arrive at $$\ddot{c}_j(t) \leq -\frac{\delta}{2} c_j(t) + \delta (s + R') + 2Q.$$
            So we can choose \emph{some} $R > s$ (uniformly) such that $c_j(t) \geq R \implies \ddot{c}_j(t) \leq \Lambda := -\frac{\delta}{2}$.
        \end{itemize}
        And so taking the max of $R$ and $\Lambda$ over both cases completes the lemma.
    \end{proof}
    The only thing that is different in this case is that we really do need $M_j$ to be uniformly small (say, $M_j \leq \frac{1}{2}$) to have any control over its acceleration. This is not an obstruction to using an analogous argument as in \Cref{thm:supportbound}: in the case of $M_j > \frac{1}{2}$ a straightforward energy argument establishes the following:
    \begin{lemma}
        Let $(\rho_t, v_t)$ be as in \Cref{thm:supportbound}, and $$\{x_1(t), \ldots, x_{N(t)}(t)\}, \{m_1, \ldots, m_{N(t)}\}, \{v_1(t), \ldots, v_{N(t)}(t)\}, \{c_1(t), \ldots, c_{N(t)}(t)\}, \{M_1, \ldots, M_{N(t)}\}$$ be as in \Cref{defn:rightcenterofmass}. Then there exists some $L > 0$ such that for all $j$ satisfying $M_j > \frac{1}{2}$, the following must be true: $$c_j(t) \leq L \qquad \text{and} \qquad \dot{c}_j(t) \leq L.$$
    \end{lemma}
    \begin{proof}
        Let $\Delta_1$ be as in \Cref{lem:momentbounds}, which gives $$|c_j(t)| \leq 2 \left| \int_{[x_j(t), \infty)} y \rho_t(dy) \right| \leq 2 \int_{[x_j(t), \infty)} |y| \rho_t(dy) \leq 2 \Delta_1.$$
        Thus, $c_j(t)$ is uniformly bounded whenever $M_j > \frac{1}{2}$. Next, we can bound $$|\dot{c}_j(t)|^2 \leq \left( 2 \int_{[x_j(t), \infty)} v_t(y) \rho_t(dy) \right)^2 \leq 4 \int_{\mathbb{R}} |v_t(y)|^2 \rho_t(dy).$$
        But the right hand side is bounded above by $8\mathcal{H}(t) \leq 8 \mathcal{H}(0)$, which is uniformly bounded in terms of $M$.
    \end{proof}
    In particular this implies that for ``massive'' clusters, the position-adjusted energy (as in \Cref{eqn:modifiedenergy}) is uniformly bounded.
    
    To run the case split in \Cref{lem:tailvelocitybound}, we need the following lemma. This is analogous to \Cref{lem:globalforcebounds}.
    \begin{lemma}\label{lem:clusterglobalforcebounds}
        Let $(\rho_t, v_t)$ be as in \Cref{thm:supportbound}, and $\{c_1(t), \ldots, c_{N(t)}(t)\}$ be as in \Cref{defn:rightcenterofmass}. Then there exists some $\overline{\Lambda} > 0$ such that $$\ddot{c}_j(t) \leq \overline{\Lambda} \qquad \forall t \quad \forall j.$$
    \end{lemma}
    \begin{proof}
        By the growth condition, there exists some $R > 0$ such that $W'(x) > 0$ for all $x \geq R$. Let $$\overline{\Lambda} := \norm{W'}_{L^\infty([-R, R])},$$
        and thus $$\ddot{c}_j(t) \leq -\fint_{[x_j(t), \infty)} \int_{(x_j(t) - R, x_j(t))} W'(x - y) \rho_t(dy) \rho_t(dx) \leq \overline{\Lambda}.$$
    \end{proof}
    For any sticky collisions, the leftmost particle has the highest velocity. This immediately yields the following observation:
    \begin{lemma} \label{lem:comatcollisions}
        Let $(\rho_t, v_t)$ be a sticky-particle solution to pressureless Euler as in \Cref{thm:supportbound}, and suppose at time $t$ the particles $x_i, \ldots, x_j$ collide, for any $i < j$. Let the state of the system before collision be $$\{x_1(t-), \ldots, x_{N(t-)}(t-)\}, \quad \{m_1, \ldots, m_{N(t-)}\}, \quad \{v_1(t-), \ldots, v_{N(t-)}(t-)\}, \quad \{c_1(t-), \ldots, c_{N(t-)}(t-)\}$$ as defined in \Cref{defn:rightcenterofmass}, and similarly let the state of the system after collision be  $$\{x_1(t), \ldots, x_{N(t)}(t)\}, \quad \{m^*_1, \ldots, m^*_{N(t)}\}, \quad \{v_1(t), \ldots, v_{N(t)}(t)\}, \quad \{c_1(t), \ldots, c_{N(t)}(t)\}$$
        That is, particles $\{x_i, \ldots, x_j\}$ collide to form the new particle $i$. Then the following holds: $$c_i(t-) = c_i(t) \qquad \text{and} \qquad \dot{c}_i(t-) = \dot{c}_i(t).$$
    \end{lemma}
    \begin{proof}
        Between collisions, $\dot{c}_i(t)$ is the total momentum of the tail divided by its mass $M_i$. At a collision at time $t$ of particles $i, \ldots, j$ (recall that we take particles ordered left-to-right, and this collision produces a new particle also indexed $i$), the sticky collision rule gives $$m_i^* v_i(t) = \sum_{k=i}^j m_kv_k(t-) \qquad m_i^* = \sum_{k=i}^j m_k.$$
        Since the total mass and momentum of the right tail are unchanged across the collision, we obtain $\dot{c}_i(t) = \dot{c}_i(t-)$. The continuity of $c_i(t)$ across collisions follows analogously.
    \end{proof}
    With these lemmas, the same proof strategy used for \Cref{thm:supportbound} and \Cref{lem:tailvelocitybound} applies, and only requires us to replace $x_j(t)$ with $c_j(t)$, and track collisions by using \Cref{lem:comatcollisions}. In particular, one can replace the position-adjusted energy \Cref{eqn:modifiedenergy} with
    \begin{equation}
        E_j(t) := \frac{1}{2} (\dot{c}_j(t))_+^2 + \Phi(c_j(t)), \qquad E(t) := \max_j [E_j(t) \cdot \mathds{1}_{[R, \infty)}(c_j(t))],
    \end{equation}
    Upon concluding that $E(t)$ is uniformly bounded for all time, the uniform boundedness of the support is immediate by the fact that $E_j(t)$ is coercive on $c_j(t)$ for all $j$, and for a system with $n$ particles the rightmost particle satisfies $c_n(t) = x_n(t)$, i.e. the rightmost center of mass is just the particle position.

\section{Conditional expectation formula for discrete solutions}
In this section we revisit the so-called ``conditional expectation formula" from \cite{Hynd2019}, which encapsulates the connection between the evolution of the Lagrangian map $X(t,y_i)$ which tracks the location at time $t$ of the particle initially sitting at location $y_i$ and the initial data $(\rho_0,v_0)$. 
Consider 
$$\rho_0:=\sum_{i=1}^nm_i\delta_{y_i},$$
where $y_1<...<y_n$, $m_i>0$, $m_1+...+m_n=1$. Set $v_i:=v_0(y_i)$ for $i=1,...,n$. As we mentioned in the introduction, while no collision occurs, the trajectory $\gamma_i$ of the particle starting at $y_i$ with initial velocity $v_i$ satisfies
\begin{equation}\label{ODEs}\ddot{\gamma}_i(t)=-\sum m_jW'(\gamma_i(t)-\gamma_j(t)), \ \gamma_i(0)=y_i,\ \dot\gamma_i(0)=v_i.
\end{equation}
 When trajectories intersect, these paths $\gamma_i$ must be modified. Due to the requirement that particles undergo perfectly inelastic collisions upon collision,  we impose that $\gamma_i\equiv\gamma_j $ for all subsequent times. Taken from \cite{Hynd2019}, the following proposition spells out some basic properties of these trajectories.
\begin{proposition}\label{gammas} There exist continuous, piecewise $C^2$ paths
$\gamma_1$,..., $\gamma_n:[0,\infty)\rightarrow\mathbb{R}$ with the following properties:\\
\noindent (i) For $i=1,...,n$ and all but finitely many $t\in(0,\infty)$, \eqref{ODEs} holds.\\
\noindent (ii) For $i=1,...,n$ we have $\gamma_i(0)=y_i$ and $\dot\gamma_i(0+)=v_i$.\\
\noindent (iii) If $\gamma_i(s)=\gamma_j(s)$ for some $s\geq0$, then $\gamma_i(t)=\gamma_j(t)$ for all $t\geq s$.\\
\noindent (iv) If $t>0$, $\{i_1,...,i_k\}\subset\{1,...,n\}$, and $\gamma_{i_1}(t)=...=\gamma_{i_k}(t)\neq\gamma_i(t)$ for $i\notin\{i_1,...,i_k\}$, then
$$\dot\gamma_{i_j}(t+)=\frac{m_{i_1}\dot\gamma_{i_1}(t-)+...+m_{i_k}\dot\gamma_{i_k}(t-)}{m_{i_1}+...+m_{i_k}}.$$
\end{proposition}
An immediate consequence of the above proposition is the {\it averaging property}, which is an assertion about the conservation of momentum:
\begin{corollary}\label{cons-momentum}
Suppose $\varphi:\mathbb{R}\rightarrow\mathbb{R}$ is a Borel function and let $0\leq s<t$. Then 
\begin{equation}\label{eq:cons-momentum}
\sum_{i=1}^nm_i\varphi(\gamma_i(t))\dot\gamma_i(t+)=\sum_{i=1}^nm_i\varphi(\gamma_i(t))\bigg[\dot\gamma_i(s+)-\int_s^t\bigg(\sum m_jW'(\gamma_i(\tau)-\gamma_j(\tau))\bigg)d\tau\bigg].
\end{equation}
\end{corollary}
An inequality, which is a consequence of the obvious fact that kinetic energy decreases through collisions, reads off
\begin{equation}\label{eq:discrete-energy}
\frac{1}{2}\sum_{i=1}^nm_i\dot\gamma_i(t+)^2+\frac{1}{2}\sum_{i,j=1}^n m_im_jW(\gamma_i(t)-\gamma_j(t))\leq \frac{1}{2}\sum_{i=1}^nm_i\dot\gamma_i(s+)^2+\frac{1}{2}\sum_{i,j=1}^n m_im_jW(\gamma_i(s)-\gamma_j(s)).
\end{equation}
 Now set $X(t,y_i)=\gamma_i(t)$ for $t\geq0$. Also, define $v$ by
 $v(t,x)=\dot\gamma_i(t+)$ if $x=\gamma_i(t)$ or $v(t,x)=0$ else. By Proposition \ref{gammas}, $v$ is well-defined and it is routine to check that $v$ is Borel as a function of $(t,x)$. Furthermore, we clearly have $\dot X(t,y)=v(t,X(t,y))$ on the support of $\rho_0$ for all but finitely many $t\geq0$. In light of this, equations \eqref{eq:cons-momentum} and \eqref{eq:discrete-energy} can be re-written as
 \begin{equation}\label{eq:cond-expect}
 \int\varphi(y)v(t,y)\rho_t(dy)=\int\varphi(X(t,y))\bigg[v_0(y)-\int_0^t(W'\ast\rho_\tau)(X(\tau,y))d\tau\bigg]\rho_0(dy)
 \end{equation}
 and 
 \begin{equation}\label{eq:nonincresing-energy}
 \frac{1}{2}\int v^2(t,y)\rho_t(dy)+\frac{1}{2}\int\!\!\!\int W(y-z)\rho_t(dy)\rho_t(dz)\leq \frac{1}{2}\int v^2(s,y)\rho_s(dy)+\frac{1}{2}\int\!\!\!\int W(y-z)\rho_s(dy)\rho_s(dz),
 \end{equation}
 respectively. Here $\rho_t:=X_{t\#}\rho_0$. Let $\chi$ denote the Lebesgue measure restricted to the unit interval $(0,1)$. If $N_0$ is the right-continuous optimal map pushing $\chi$ forward to $\rho_0$, we get $\rho_t=X(t,N_0)_\#\chi$, and since $X_t$ and $N_0$ are monotone nondecreasing, we infer that $N_t:=X_t\circ N_0$ is the right-continuous optimal map pushing $\chi$ forward to $\rho_t$. 

 Since our potential $W$ is not assumed to be $c$-convex, the so-called ``quantitative stickiness" property \cite{Hynd2019} that 
 $$\frac{|\gamma_i(t)-\gamma_j(t)|}{\mathrm{sinh}(\sqrt{c}t)}\leq \frac{|\gamma_i(s)-\gamma_j(s)|}{\mathrm{sinh}(\sqrt{c}s)}\mbox{ for all }0<s\leq t$$
 is not available to us. In \cite{Hynd2019} this estimate is used to prove that the passage to the limit of the approximating maps $X^n$ (built for the discrete approximations $\rho^n_0$ to $\rho_0$) provides a limiting map $X$ that satisfies the same estimate and the pair $(\rho,v)$ solves the system. In particular, $\rho_t=X_{t\#}\rho_0$, so the solution at time $t>0$ is the push forward of $\rho_0$ through the map $X_t$. In fact, in \cite{Hynd2019} it is shown that there exist Lipschitz maps $f_{s,t}$ such that $f_{s,t}\circ X_s=X_t$ if $0\leq s\leq t$.  With $c$-convexity out of the picture, we are not able to reach similar conclusions. In particular, we do not know how to obtain a limiting map $X$, which ultimately prevents us from constructing a solution $(\rho,v)$ by setting $\rho_t=X_{t\#}\rho_0$. However, if we forgo the goal of constructing a solution that satisfies a stickiness condition of any kind, we can do that by passing to the limit at the level of the maps $N^n_t:=X^n_t\circ N^n_0$, which push forward $\chi$ to $\rho^n_t$. The tradeoff is that the limiting maps $N_t$ are not necessarily compositions with $N_0$, i.e. we will not obtain a map $X_t$ such that $N_t=X_{t}\circ N_0$. 

 It is therefore advantageous to our approach to rewrite \eqref{eq:cond-expect} as
 \begin{equation}\label{eq:cond-expect-N}
 \int_0^1\varphi(N(t,x))\dot N(t,x)dx=\int_0^1\varphi(N(t,x))\bigg[v_0(N_0(x))-\int_0^t\int_0^1W'(N(s,x)-N(s,z))dzds\bigg]dx.
 \end{equation}

\section{Existence of solutions from compactly supported initial data}
   \begin{theorem} \label{thm:existence}
        Let $W \in C^1(\R)$ be an even potential that satisfies the globally-in-time uniform boundedness condition, i.e. the sticky particle solution (SPS) starting at a discrete measure in $\mathcal{P}(\R)$ supported in $[-R, R]$ with initial velocity $v_0 \in C(\R)$ satisfies
        \begin{enumerate}
            \item $\supp \rho_t \subset [-M, M]$
            \item $\norm{v_t}_{L^\infty(\rho_t)} \leq M$
        \end{enumerate}
        for all $t \geq 0$ and some $M \in \R$. Let $\rho_0 \in \mathcal{P}(\R)$ be supported in $[-R, R]$. Then there exists a distributional solution $(\rho_t, v_t)$ starting at $(\rho_0, v_0)$ that obeys the same bounds as above.
    \end{theorem}

    \begin{proof}
        Let $\{\rho_0^n\}_n \subseteq \mathcal{P}(\R)$ be a sequence of discrete measures which converges to $\rho_0$ in the $W_2$ metric. Denote by $(\rho_t^n, v_t^n)$ the resulting (SPS). By the uniform (in $t$ and $n$) bound on the support of $\rho_t^n$, we can extract, via a standard diagonal argument, a subsequence $\rho_t^{n_k}$ such that $\rho_t^{n_k} \xrightarrow[k \to \infty]{W_2} \rho_t$ for some $\rho_t \in \mathcal{P}(\R)$ and all $t \in \Q \cap [0, \infty)$. Since 
        \[
            W_2(\rho_t^n, \rho_s^n) \leq \int_s^t \norm{v_\tau^n}_{L^2(\rho_\tau^n)} d\tau \qquad \text{for all } n \text{ and } 0 \leq s < t < \infty\text{,}
        \]
        we use the uniform bound $\norm{v_t^n}_{L^\infty(\rho_t^n)} \leq M$ to deduce that
        \[
            W_2(\rho_t^n, \rho_s^n) \leq M|t - s| \qquad \text{ for all } n \text{ and } 0 \leq s < t < \infty\text{.}
        \]
        Let $t \in (0, \infty)$. Consider any sequence of rationals $\{t_j\}_{j \geq 1}$ such that $t_j \xrightarrow[j \to \infty]{} t$. Since $W_2(\rho_{t_i}^{n_k}, \rho_{t_j}^{n_k}) \leq M|t_i - t_j|$ for all $i, j \in \N$, we see that $\{\rho_{t_i}^{n_k}\}_{i \geq 1}$ is a Cauchy sequence in $(\mathcal{P}_2(\R), W_2)$. Since $(\mathcal{P}_2(\R), W_2)$ is a complete metric space, there exists $\rho_t^{n_k} \in \mathcal{P}(\R)$ such that $\lim_{j \to \infty} W_2(\rho_{t_j}^{n_k}, \rho_{t}^{n_k}) = 0$. Take a different sequence $\{\overline{t}_j\}_{j \geq 1}$ with $\overline{t}_j \xrightarrow[j \to \infty]{} t$. By constructing a new sequence that alternates between $t_i$ and $\overline{t}_i$, we observe that the limit $\rho_{t}^{n_k}$ is independent of the choice of the sequence $\{t_j\}_{j \geq 1}$.

        By the triangle inequality,
        \[
            W_2(\rho_{t}^{n_k}, \rho_{t}^{n_j}) \leq W_2(\rho_{t}^{n_k}, \rho_{t_i}^{n_k}) + W_2(\rho_{t_i}^{n_k}, \rho_{t_i}^{n_j}) + W_2(\rho_{t_i}^{n_j}, \rho_{t}^{n_j}) \leq 2M|t_i - t| + W_2(\rho_{t_i}^{n_k}, \rho_{t_i}^{n_j})\text{.}
        \]
        Let $\varepsilon > 0$ and choose $t_{i_0} \in \Q \cap [0, \infty)$ such that $|t_{i_0} - t| < \varepsilon/(4M)$. For sufficiently large $n_j$ and $n_k$, we have
        $W_2(\rho_{t_{i_0}}^{n_k}, \rho_{t_{i_0}}^{n_j}) \leq \varepsilon/2$ and since $2M|t_{i_0} - t| < \varepsilon/2$, 
        it follows that $W_2(\rho_{t}^{n_k}, \rho_{t}^{n_j}) < \varepsilon$. Thus, $\{\rho_t^{n_k}\}_{k \geq 1}$ is a Cauchy sequence. Again, since $(\mathcal{P}_2(\R), W_2)$ is complete, we have $\rho_t \in \mathcal{P}(\R)$ such that $\lim_{k \to \infty} W_2(\rho_{t}^{n_k}, \rho_t) = 0$. Therefore, we have a sequence $\rho_{t}^{n_k}$ such that
        \[
            \lim_{k \to \infty} W_2(\rho_{t}^{n_k}, \rho_t) = 0 \qquad \text{for all } t \in [0, \infty)\text{.}
        \]

        Let $N_t^n$ and $N_t$ be the optimal maps such that $N_{t\#}^n\chi = \rho_t^n$ and $N_{t\#}\chi = \rho_t$ for all $n \geq 1$ and $t \geq 0$, where $\chi = \mathcal{L}|_{[0, 1]}$. But $W_2(\rho_{t}^{n_k}, \rho_t) = \norm{N_{t}^{n_k} - N_t}_{L^2(0, 1)}$, so
        \begin{align}
            \lim_{k \to \infty} \norm{N_{t}^{n_k} - N_t} = 0\text{.} \label{eq:lim}
        \end{align}
        Since
        \[
            \mathrm{spt} (\rho_{t}^{n_k}), \mathrm{spt} (\rho_t) \subseteq [-M, M]\text{,}
        \]
        we have
        \begin{equation}\label{full-cv}
            \norm{N_{t}^{n_k}}_{L^\infty(0, 1)} \leq M\text{,} \qquad \norm{N_t}_{L^\infty(0, 1)} \leq M\text{.}
        \end{equation}
        Hence, by the dominated convergence theorem, \eqref{eq:lim} implies
        \[
            \int_0^T \norm{N_{t}^{n_k} - N_t}^2dt \xrightarrow[k \to \infty]{} 0 \qquad \text{for all } T > 0\text{.}
        \]

        We next consider
        \[
            \int_0^T\int_0^1\int_0^1 |N_{t}^{n_k}(x) - N_t(x) - N_{t}^{n_k}(z) + N_t(z)|^2dzdxdt \leq 4\int_0^T \norm{N_{t}^{n_k} - N_t}^2dt \xrightarrow[k \to \infty]{} 0\text{,}
        \]
        consequently deducing that there exists $A \subseteq [0, T] \times (0, 1) \times (0, 1)$ such that $\mathcal{L}^3(A) = T$ (full measure) and a subsequence of $\{n_k\}_{k \geq 1}$, not relabeled, such that 
        \[
            N_{t}^{n_k}(x) - N_{t}^{n_k}(z) \xrightarrow[k \to \infty]{} N_t(x) - N_t(z) \qquad \text{for all } (t, x, z) \in A\text{.}
        \]
        This implies
        \begin{align}
            W'(N_{t}^{n_k}(x) - N_{t}^{n_k}(z)) \xrightarrow[k \to \infty]{} W'(N_t(x) - N_t(z)) \qquad \text{for all } (t, x, z) \in A\text{.} \label{eq:wprimeina}
        \end{align}

        Now, let $\varphi \in C^2(\R)$ and denote
        \begin{align*}
            J_k &= \int_0^T\int_0^1\int_0^1 \varphi'(N_{T}^{n_k}(x))W'(N_{t}^{n_k}(x) - N_{t}^{n_k}(z))dzdxdt,
            \\ 
            J &= \int_0^T\int_0^1\int_0^1 \varphi'(N_T(x))W'(N_t(x) - N_t(z))dzdxdt\text{.}
        \end{align*}

     By the oddness of $W'$,
        \begin{align*}
            J &= \int_0^T\int_0^1\int_0^1 \varphi'(N_T(z))W'(N_t(z) - N_t(x))dzdxdt \\
            &= -\int_0^T\int_0^1\int_0^1 \varphi'(N_T(z))W'(N_t(x)-N_t(z))dzdxdt\text{,}
        \end{align*}
        so
        \[
            2J = \int_0^T\int_0^1\int_0^1 \left [ \varphi'(N_T(x)) - \varphi'(N_T(z))\right ] W'(N_t(x) - N_t(z))dzdxdt\text{.}
        \]
        Likewise,
        \[
            2J_k = \int_0^T\int_0^1\int_0^1 \left [ \varphi'(N_T^{n_k}(x)) - \varphi'(N_T^{n_k}(z))\right]W'(N_t^{n_k}(x) - N_t^{n_k}(z))dzdxdt\text{.}
        \]
        Combining these, we get that $2(J_k-J)$ is equal to
        \begin{align*} \label{eq:firstint}
             & \int_0^T\int_0^1\int_0^1 \left[\varphi'(N_{T}^{n_k}(x)) - \varphi'(N_{T}^{n_k}(z)) - \varphi'(N_T(x)) + \varphi'(N_T(z))\right]W'(N_{t}^{n_k}(x) - N_{t}^{n_k}(z))dzdxdt \\
            &+ \int_0^T\int_0^1\int_0^1 \left[\varphi'(N_T(x)) - \varphi'(N_T(z))\right]\left[W'(N_{t}^{n_k}(x) - N_{t}^{n_k}(z)) - W'(N_t(x) - N_t(z))\right]dzdxdt\text{.} 
        \end{align*}
        Since
        \[
            |\varphi'(N_{T}^{n_k}(x)) - \varphi'(N_T(x))| \leq \norm{\varphi''}_{L^\infty(-M, M)}|N_{T}^{n_k}(x) - N_T(x)|\text{,}
        \]
        the first integral in the expression above for $2(J_k-J)$ is bounded above in absolute value by
        \begin{align*}
            &2T\norm{W'}_{L^\infty(-2M, 2M)}\norm{\varphi''}_{L^\infty(-M, M)}\int_0^1|N_{T}^{n_k}(x) - N_T(x)|dx \\
            &\leq 2T\norm{W'}_{L^\infty(-2M, 2M)}\norm{\varphi''}_{L^\infty(-M, M)}\norm{N_{T}^{n_k} - N_T}_{L^2(0, 1)} \\
            &= 2T\norm{W'}_{L^\infty(-2M, 2M)}\norm{\varphi''}_{L^\infty(-M, M)}W_2(\rho_{T}^{n_k}, \rho_T)\text{.}
        \end{align*}
        For the second integral in that expression, we use \eqref{eq:wprimeina} and the dominated convergence theorem to conclude that it converges to $0$ as $k \to \infty$. Therefore, 
        \begin{align}
            \lim_{k \to \infty} J_k = J\text{.} \label{eq:limj}
        \end{align}
        
        Furthermore, the uniform (w.r.t. $t$ and $n$) bound on $\|v^n_t\|_{L^\infty(\rho^n_t)}$ implies a uniform (w.r.t. $t$ and $n$) bound on $\|v^n_t\|_{L^2(\rho^n_t)}$, which in turn implies that $$\sup_n \|\dot N^n_t\|_{L^2(0,1)} < \infty \qquad \forall T > 0.$$
        Thus, there exists some $U \in L^2((0, T) \times (0, 1))$ such that $$\dot{N}^{n_k} \xrightharpoonup[k \to \infty]{} U \qquad \text{in } L^2((0, T) \times (0, 1)),$$ after passing to a further subsequence (which we will not notate). Since $N^{n_k} \rightarrow N$ strongly in that same space, we can identify $\dot{N} = U$, and thus have that $N \in H^1(0, T; L^2(0, 1))$. By a diagonal argument (which we will not notate) we can ensure that this holds on any finite interval, i.e. on $[0, T]$ for all $T$.
        
        Now, note that by \eqref{eq:cond-expect-N} we have
        \begin{equation} \label{momentum-N-k}
            \frac{d}{dt}\int_0^1 \varphi(N_{t}^{n_k}(x))dx = \int_0^1 \varphi'(N_{t}^{n_k}(x))\left[v_0(N_{0}^{n_k}(x)) - \int_0^t\int_0^1 W'(N_{s}^{n_k}(x) - N_{s}^{n_k}(z))dzds\right]dx.
        \end{equation}
        Set
        \[
        F_k(t):=\int_0^1\varphi(N_t^{n_k}(x))\,dx,
        \qquad
        F(t):=\int_0^1\varphi(N_t(x))\,dx,
        \]
        and denote the right hand sides of \eqref{momentum-N-k} and of \eqref{momentum-N} by $G_k(t)$ and $G(t)$, respectively. Then $F_k'=G_k$ a.e. on $[0,T]$. For every $t\in[0,T]$, we have
        \[
        F_k(t) \longrightarrow F(t) \qquad G_k(t) \longrightarrow G(t),
        \]
        where the latter follows from \eqref{eq:limj}, applied with $t$ in place of $T$, and the convergence of $v_0(N_0^{n_k})$. Moreover,
        \[
        \sup_k\|G_k\|_{L^\infty(0,T)}
        \le
        \|\varphi'\|_{L^\infty(-M,M)}
        \left(\|v_0\|_{L^\infty(-R,R)}
        +T\|W'\|_{L^\infty(-2M,2M)}\right)<\infty.
        \]
        The fundamental theorem of calculus and dominated convergence therefore yield $F'=G$ a.e. on $[0,T]$, which is precisely \eqref{momentum-N}:
        \begin{equation}\label{momentum-N}
            \frac{d}{dt}\int_0^1 \varphi(N_t(x))dx = \int_0^1 \varphi'(N_t(x))\left[v_0(N_0(x)) - \int_0^t\int_0^1 W'(N_s(x) - N_s(z))dzds\right]dx
        \end{equation}
        for $\mathcal{L}-\text{a.e. } t \in [0, T]$.
        This identity, as in \cite{Hynd2019}, encodes the information that is sufficient to conclude that $(\rho_t, v_t)$, where $v_t$ is the unique function in $L^2(\rho_t)$ such that $\dot{N_t} = v_t \circ N_t$ as in \cite{Tudorascu2019}, solves our system with initial data $(\rho_0, v_0)$. Indeed, the last displayed identity shows 
        \[
            [0,T]\ni t\rightarrow \int_0^1\varphi(N_t(x))\dot N_t(x)dx \text{ lies in } W^{1,1}(0,T)
        \]
        and 
        \begin{eqnarray}\label{momentum-eqn}
            \frac{d}{dt}\int_0^1 \varphi(N_t(x))\dot N_t(x)dx &=& \int_0^1 \varphi'(N_t(x))\dot N_t(x)\left[v_0(N_0(x)) - \int_0^t\int_0^1 W'(N_s(x) - N_s(z))dzds\right]dx\nonumber\\
            &-&\int_0^1 \varphi(N_t(x))\int_0^1 W'(N_t(x) - N_t(z))dzdx.
        \end{eqnarray}
        
        But 
        \begin{eqnarray*}
            \int_0^1 \varphi'(N_t(x))\dot N_t(x)\left[v_0(N_0(x)) - \int_0^t\int_0^1 W'(N_s(x) - N_s(z))dzds\right]dx=\int_0^1 \varphi'(N_t(x))|\dot N_t(x)|^2dx\text{,}
        \end{eqnarray*}
        which is obtained by first approximating $v_t$ in $L^2(\rho_t)$ strongly by smooth and compactly supported functions $\psi$ and noting that  
        \begin{eqnarray*}
            \int_0^1 \varphi'(N_t(x))\psi(N_t(x))\left[v_0(N_0(x)) - \int_0^t\int_0^1 W'(N_s(x) - N_s(z))dzds\right]dx\\
            =\int_0^1 \varphi'(N_t(x))\psi(N_t(x))\dot N_t(x)dx.
        \end{eqnarray*}
    This last equation exhibits \eqref{momentum-N} applied to an antiderivative of $\varphi'\psi$ instead of $\varphi$.
    Finally, if we use the identity displayed after \eqref{momentum-eqn} and rewrite \eqref{momentum-eqn} in terms of $\rho_t$ and $v_t$ (by taking into account $\dot N_t=v_t\circ N_t$ \cite{Tudorascu2019} and $\rho_t=N_{t\#}\chi$), we get the momentum equation.
    \end{proof}

    \begin{remark}
        By \Cref{thm:supportbound} and \Cref{corr:velocitybound}, any admissible potential (in the sense of \Cref{defn:admissiblepotential}) will satisfy the global-in-time uniform boundedness condition required in \Cref{thm:existence}.

        We also observe that a $C^1$ admissible potential need not be $c$-convex: for example, consider the double-welled potential $W(x) = |x|^2 - |x|^{4/3}$. We note that such an example is not even locally $c$-convex.
    \end{remark}

\section{Asymptotic examples for general potentials}
    In \cite{Tudorascu2024REU} the authors showed that for the purely attractive potential $W(x)=|x|$, the SPS originating in a compact set collapses into the Dirac delta originating at the initial distribution's center of mass, moving with constant velocity equal to the weighted average of the initial velocities. Additionally, it was proven that the time to collapse can be explicitly written in terms of $\rho_0$ and $v_0$. Similarly, our admissible potentials include long range attractiveness and thus we briefly explore what can be expected in terms of the asymptotic behavior of solutions.    

    For a smooth even potential and initial configuration of two particles of total mass one, the separation distance between the two particles $d(t)$ satisfies the differential equation 
    \begin{equation}
    \ddot d(t)=-W'(d(t)).\label{eq:smooth-separation}
    \end{equation}
    Thus, any nonzero critical point $d_*$ of $W$ yields a noncollapsed stationary two-particle solution at separation $d_*$. 
    \begin{example}(Unique global minimum is insufficient.)
    Let
    \begin{equation}
    W_1(x)=x^6-3x^4+3x^2=1+(x^2-1)^3.\label{eq:poly1}
    \end{equation}
    Then \(W_1\) is even, smooth, coercive, and has a unique global minimum. However,
    $W_1'(x)=6x(x^2-1)^2$, 
    so \(W_1'(1)=0\). Thus, two particles at separation one with zero relative velocity remain at that separation forever. Therefore an even coercive potential with unique global minimum at zero need not force collapse to a single Dirac mass.
    \end{example}

    The nonzero critical point here is degenerate and not a local minimum. The next example shows that even adding attraction near zero and excluding nearby nonzero critical points is insufficient.
    \begin{example}(Coercivity, near-zero attraction, and a nonzero local well)
    Let
    \begin{equation}
    W_2(x)=\frac{x^6}{6}-\frac{13x^4}{16}+\frac{9x^2}{8}.\label{eq:poly2}
    \end{equation}
    Then
    \begin{equation}
    W_2'(x)=x(x^2-1)\left(x^2-\frac94\right).\label{eq:poly2-der}
    \end{equation}
    The potential is even, smooth, and coercive and zero is the unique global minimizer.

    For \(x>0\), the force derivative has signs
    \begin{equation}
    W_2'(x)
    \begin{cases}
    >0,&0<x<1,\\
    <0,&1<x<3/2,\\
    >0,&x>3/2.
    \end{cases}\label{eq:poly2-sign}
    \end{equation}
    Thus the interaction is attractive at small positive separations, has no nonzero critical point in $(0,1)$, and has a nonzero local minimum at $d=3/2$. Two particles placed at separation $3/2$ with zero relative velocity remain there and do not collapse.
    \end{example}

    This disproves the sufficiency of the combined assumptions: evenness, coercivity, unique global minimum at zero, and attraction throughout a neighborhood of zero.
    We also have noncollapsing solutions that aren't stationary:
    \begin{example} (A non-stationary, noncollapsing example)
    \begin{equation}
    W_3(x)=\frac{x^6}{6}-\frac{3x^4}{4}+x^2,
    \quad
    W_3'(x)=x(x^2-1)(x^2-2).\label{eq:poly3}
    \end{equation}
    This potential is even, coercive, and has a unique global minimum at zero. By choosing the initial relative speed so that the two-particle energy equals the value at the nonzero critical separation $d=1$, one obtains an orbit approaching $d=1$ with vanishing relative velocity rather than collapsing. 
    \end{example}

    \pgfplotsset{
        potential plot/.style={
            width=\textwidth,
            height=0.75\textwidth,
            axis lines=middle,
            xlabel={$x$},
            ylabel={$y$},
            ymin=-2, ymax=10,
            samples=200,
            grid=both,
            enlarge x limits=0.05,
            enlarge y limits=0.05,
            legend style={
                at={(0.5,0.97)}, 
                anchor=north,
                font=\footnotesize,
                inner sep=2pt,
                row sep=0pt,
                column sep=2pt,
            },
            legend image post style={xscale=0.7},
            tick label style={font=\footnotesize},
            label style={font=\small},
            xlabel style={anchor=west},
            ylabel style={anchor=south},
        }
    }

    \begin{figure}[H]
        \centering

        \begin{minipage}[t]{0.32\textwidth}
            \centering
            \begin{tikzpicture}
                \begin{axis}[
                    potential plot,
                    xmin=-2, xmax=2,
                    domain=-2:2,
                    xtick={-2,-1,1,2},
                    ytick={5,10}
                ]
                    \addplot[red, thick] {x^6 - 3*x^4 + 3*x^2};
                    \addplot[blue, thick] {6*x*(x^2-1)^2};
                \end{axis}
            \end{tikzpicture}

            \par\smallskip
            \textbf{$W_1$}
        \end{minipage}
        \hfill
        \begin{minipage}[t]{0.32\textwidth}
            \centering
            \begin{tikzpicture}
                \begin{axis}[
                    potential plot,
                    xmin=-3, xmax=3,
                    domain=-2.5:2.5,
                    xtick={-2,2},
                    ytick={5,10}
                ]
                    \addplot[red, thick]
                        {(1/6)*x^6 - (13/16)*x^4 + (9/8)*x^2};
                    \addplot[blue, thick]
                        {x*(x^2-1)*(x^2-9/4)};
                \end{axis}
            \end{tikzpicture}

            \par\smallskip
            \textbf{$W_2$}
        \end{minipage}
        \hfill
        \begin{minipage}[t]{0.32\textwidth}
            \centering
            \begin{tikzpicture}
                \begin{axis}[
                    potential plot,
                    xmin=-3, xmax=3,
                    domain=-2.5:2.5,
                    xtick={-2,2},
                    ytick={5,10}
                ]
                    \addplot[red, thick]
                        {(1/6)*x^6 - (3/4)*x^4 + x^2};
                    \addplot[blue, thick]
                        {x*(x^2-1)*(x^2-2)};
                \end{axis}
            \end{tikzpicture}

            \par\smallskip
            \textbf{$W_3$}
        \end{minipage}

        \caption{The potentials $W_i$ (red) and their derivatives $W_i'$ (blue).}
        \label{fig:poly-comparison}
    \end{figure}
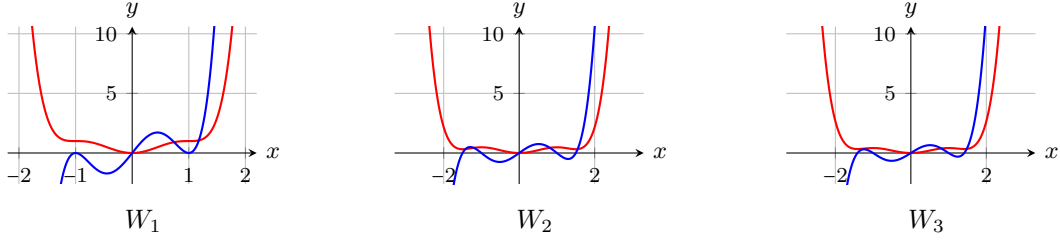

\appendix
\section{Application of confinement to $c$-convex potentials} \label{app:semiconvexasymptote}
    We conclude by discussing an application of \Cref{thm:supportbound} and \Cref{corr:velocitybound} in the setting of $c$-convex potentials. 
    We will show for the solutions of \cite{Hynd2020} started in a compact interval with absolutely continuous velocity, that the results in \Cref{sec:uniformconfinement} allow us to take limits of time-shifted solutions as the shift tends to infinity, and that such limits are constant-Hamiltonian \emph{eternal solutions}, i.e. solutions defined for all $t \in \mathbb{R}$. 
    
    As an application, we show that the asymptotic statement \Cref{corr:convergencetodelta} is immediate. Thus, the confinement estimates in \Cref{sec:uniformconfinement} provide additional information even in the $c$-convex setting.
    
    To avoid contending with questions of uniqueness, we will consider the following class of solutions:
    \begin{definition}[Trajectory map solutions \cite{Hynd2020}] \label{defn:trajectorymapsoln}
        Fix an admissible potential $W$ that is also $C^1$ and $c$-convex. For $\rho_0 \in \mathcal{P}(\mathbb{R})$ compactly supported and $v_0$ absolutely continuous, by \cite[Theorem 1.1, Corollary 4.4]{Hynd2020} there exists a SPS $(\rho_t, v_t)$ corresponding to this initial data, with an associated trajectory map $X(t, x)$. We say that such a triple $(\rho_t, v_t, X(t, \cdot))$ is a trajectory map solution, and without loss of generality require that $\int x d \rho_0 = \int v_0 d \rho_0 = 0$.
    \end{definition}
    \begin{remark}
        Unlike \cite{Hynd2020}, for a trajectory map solution we do not require that the potential $W$ has derivative growing at most linearly. By our assumptions on initial data, \Cref{thm:supportbound} and \Cref{corr:velocitybound} apply to the discrete particle approximations used in the proof of \cite[Theorem 1.1]{Hynd2020}, which are confined to $[-C, C]$ for some $C > 0$. So we can use an even cutoff function to make $x \mapsto \sgn(x) W'(x)$ constant on $\mathbb{R} \setminus [-4C, 4C]$ without modifying $W$ on $[-2C, 2C]$, and just work with this new $W$ instead.
    \end{remark}

    The following objects will be important in the sequel. For $Y \in L^\infty(0, 1)$ put $$\mathcal{F}[Y](a) := -\int_0^1 W'(Y(a) - Y(b)) db \qquad \mathcal{P}[Y] := \frac{1}{2} \int_0^1 \int_0^1 W(Y(a) - Y(b)) \, da \, db.$$
    When the space is not specified $\norm{\cdot}$ and $\langle \cdot, \cdot \rangle$ will denote the norm and inner product on $L^2(0, 1)$, respectively.

    The following definition will be important, as they are the limiting objects addressed in this section:
    \begin{definition}[Eternal solutions] \label{defn:perfecteternal}
        We say that $Y \in C^2(\mathbb{R}; L^2(0, 1)) \cap L^\infty(\mathbb{R} \times (0, 1))$ is an \emph{eternal solution} if:
        \begin{enumerate}
            \item For every $t \in \mathbb{R}$, $Y_t$ is nondecreasing
            \item For all $t$, we have that $$\ddot{Y}_t = \mathcal{F}[Y_t] := -\int_0^1 W'(Y_t(a) - Y_t(b)) db.$$
        \end{enumerate}
        Note that by \Cref{lem:forcepotential} below, $\frac{1}{2} \norm{\dot{Y}_t}^2 + \mathcal{P}[Y_t]$ is constant; in other words the Hamiltonian of an eternal solution must be constant.
    \end{definition}    
    We first prove the following technical lemma:
    \begin{lemma} \label{lem:forcepotential}
        Let $R > 0$ and $L_R := \norm{W'}_{L^\infty(-2R, 2R)}$. The following assertions hold:
        \begin{enumerate}
            \item Suppose $Y, Z \in L^2(0, 1)$ satisfies $\norm{Y}_{L^\infty}, \norm{Z}_{L^\infty} \leq R$. 
            Then $|\mathcal{F}[Y]| \leq L_R$ and $$|\mathcal{P}[Y] - \mathcal{P}[Z]| \leq L_R \norm{Y - Z}_{L^1}.$$
            \item Let $Y^n, Y \in L^2(0, 1)$. If $Y^n \rightarrow Y$ in $L^2(0, 1)$ with $\norm{Y^n}_{L^\infty} \leq R$, then $\mathcal{F}[Y^n] \rightarrow \mathcal{F}[Y]$ in $L^2(0, 1)$.
            \item Let $J \subset \mathbb{R}$ be an interval and let $J \ni t \mapsto Y_t \in L^2(0, 1)$ be a map such that $\norm{Y_t}_{L^\infty} \leq R$ for all $t$. If this map is differentiable at $t \in J$ then $\tau \mapsto \mathcal{P}[Y_\tau]$ is also differentiable at $t$, and moreover we can recognize the derivative as $$\frac{d}{d\tau}\Big|_{\tau = t} \mathcal{P}[Y_\tau] = -\langle \mathcal{F}[Y_t], \dot{Y}_t \rangle.$$
            Here, differentiability is understood in the sense of maps into $L^2(0, 1)$.
        \end{enumerate}
    \end{lemma}
    \begin{proof}
        \thmendspace
        Item (1) is immediate by the following bound: $$|W(Y(a) - Y(b)) - W(Z(a) - Z(b))| \leq L_R \left( |Y(a) - Z(a)| + |Y(b) - Z(b)| \right).$$
        From item (1), we can obtain (2) by dominated convergence, using the fact that $W$ is $C^1$ and passing to a subsequence where $Y^n \rightarrow Y$ pointwise almost everywhere.

        For item (3), let $\omega(r)$ be the modulus of continuity of $W'$ on $[-2R, 2R]$. Note that for any $x, h$ such that both $x$ and $x + h$ are in $[-2R, 2R]$ we have $$W(x + h) - W(x) = W'(x) h + \int_0^1 (W'(x + \ell h) - W'(x)) h \cdot d \ell,$$
        where the last integral is bounded by $|h| \omega(|h|)$. For any $\tau \in J$ set $H = Y_\tau - Y_t$, and because for any $\delta > 0$ we can write $|h| \omega(|h|) \leq \omega(\delta) |h| + 2 L_R \delta^{-1} |h|^2$, we get $$\mathcal{P}[Y_\tau] - \mathcal{P}[Y_t] = \frac{1}{2} \iint W'(Y_t(a) - Y_t(b)) (H(a) - H(b)) \,da\, d b + o(\norm{H}_{L^2(0, 1)}).$$ 
        By oddness of $W'$, we can simplify the double integral, which gives $$\mathcal{P}[Y_\tau] - \mathcal{P}[Y_t] = -\langle \mathcal{F}[Y_t], H \rangle + o(\norm{H}_{L^2(0, 1)}).$$
        from which item (3) is immediate by chain rule.
    \end{proof}
    We can now state the main result of this section:
    \begin{theorem} \label{thm:shiftedlimits}
        Consider a potential $W$ and a trajectory map solution $(\rho_t, v_t, X(t, \cdot))$ consistent with \Cref{defn:trajectorymapsoln}. Set $N_0$ to be the optimal map such that $(N_0)_\# \chi = \rho_0$, and for any $t \geq 0$ define $N_t := X(t, N_0(\cdot))$.

        Then there is a set $G \subset (0, \infty)$ of full measure such that $\mathcal{H}$ (see \eqref{eq:hamiltonian}) is well-defined and nonincreasing on $G$. 
        Moreover, let $\mathcal{H}_\infty := \lim_{t \rightarrow \infty, t \in G} \mathcal{H}(t)$, and for every sequence $\{t_n\} \subset G$ with $t_n \rightarrow \infty$ there exists up to a subsequence (which we will suppress) an eternal solution $Y$ such that $$N_{t_n + \bullet} \rightarrow Y \quad \text{in } C([-T, T]; L^2(0, 1)) \quad \forall T > 0 \quad \text{and} \quad \dot{N}_{t_n} \rightarrow \dot{Y}_0 \quad \text{strongly in } L^2(0, 1),$$
        and $\frac{1}{2} \norm{\dot{Y}_t}^2 + \mathcal{P}[Y_t] = \mathcal{H}_\infty$ for every $t \in \mathbb{R}$.
    \end{theorem}
    \begin{remark}
        Note that the solution operator for pressureless Euler with such a potential may not have a semigroup property. As such, \Cref{thm:shiftedlimits} must be stated in terms of limits of shifted trajectories, instead of ``restarting'' the PDE at a sequence of times. Also note that the limiting $Y$ solves an equation different from the original trajectory map solution.
    \end{remark}

    \begin{framed}
        Fix once and for all a potential $W$, a trajectory map solution $(\rho_t, v_t, X(t, \cdot))$, and $N_t$, as in \Cref{thm:shiftedlimits}. It remains to construct $G$ and establish all the conclusions of \Cref{thm:shiftedlimits}.
    \end{framed}
    By \Cref{thm:supportbound} and \Cref{corr:velocitybound}, there exists some $M > 0$ such that
    \begin{equation} \label{eqn:appendixsupportbounds}
        \mathrm{spt} (\rho_t) \subset [-M, M] \qquad \norm{v_t}_{L^\infty(\rho_t)} \leq M \qquad \forall t \geq 0.
    \end{equation}
    This is because the solutions in \cite{Hynd2020} are limits of discrete SPS, for which the uniform bounds in \Cref{sec:uniformconfinement} apply, and thus pass to the limit.

    Put $V_0 := v_0 \circ N_0$, and let $L := L_M$ be as in \Cref{lem:forcepotential}, so that $|\mathcal{F}[N_t]| \leq L$. 
    Let $\mathcal{M}$ denote the space of non-decreasing functions on $(0, 1)$ bounded in the $L^\infty$ norm by $M$. Finally, for $t \geq 0$ let $$\mathcal{K}_t := \{h \circ N_t : h \text{ Borel}, h \circ N_t \in L^2(0, 1)\},$$
    which is a closed subspace of $L^2(0, 1)$ that is isometric to $L^2(\rho_t)$, and so there is a well-defined orthogonal projection $\mathrm{Proj}_{\mathcal{K}_t}: L^2(0, 1) \rightarrow \mathcal{K}_t$.
    \begin{lemma} \label{lem:goodtimes}
        There exists a set $G \subset (0, \infty)$ of full measure such that if $t \in G$ and $k \in \mathbb{Z}_{\geq 0}$ satisfying $k < t$, then $t - k \in G$. And additionally, for every $t \in G$ we have:
        \begin{itemize}
            \item $s \mapsto N_s$ is differentiable at $s=t$ as a map into $L^2(0, 1)$ with $\dot{N}_t = v_t \circ N_t$, and $\norm{\dot{N}_t}_{L^\infty} \leq M$
            \item $\dot{N}_t = \mathrm{Proj}_{\mathcal{K}_t} \left[ V_0 + \int_0^t \mathcal{F}[N_\tau] d\tau \right]$
            \item For almost every $a, b$ with $a < b$ we have $$\dot{N}_t(b) - \dot{N}_t(a) \leq \frac{\sqrt{c}}{\tanh(\sqrt{c} t)} (N_t(b) - N_t(a)).$$
        \end{itemize}
    \end{lemma}
    \begin{proof}
        \thmendspace
        Let $G_0$ be the set of $t > 0$ satisfying the last three properties. Then $G_0$ has full measure, by \cite[Theorem 1.1, Equation (1.3), Corollary 4.4]{Hynd2020}. We just need to close under integer shifts. To that end, define $$G := \{t \in G_0 : t - k \in G_0 \quad \forall k \in \mathbb{Z} \cap (0, t)\}.$$
        Then $G$ is full measure and satisfies all the requirements.
    \end{proof}
    Going forward, we will fix $G$ to be the set of good times, as in \Cref{lem:goodtimes}. 

    We know that $t \mapsto N_t$ is locally Lipschitz into $L^2(0, 1)$, and thus for all $s, t \in G$ we have $N_t - N_s = \int_s^t \dot{N}_\tau d \tau$. So by earlier velocity bounds, this also gives 
    \begin{equation} \label{eqn:lipschitzofN}
        \norm{N_t - N_s}_{L^2(0, 1)} \leq M |t - s| \qquad \forall s, t \in G.
    \end{equation}
    
    Also, for $t \in G$ we have $\mathcal{H}(t) = \frac{1}{2} \norm{\dot{N}_t}^2 + \mathcal{P}[N_t]$, as $(N_t)_\# \chi = \rho_t$ and $\dot{N}_t = v_t \circ N_t$. This will be of use later.

    We first show that along sequences in $G$ (i.e. ``generic'' sequences) we can obtain \emph{strong} subsequential limits of the velocity.
    \begin{lemma} \label{lem:strongvelocityconvergence}
        For every sequence $\{t_n\} \subset G$ with $t_n \rightarrow \infty$, up to a subsequence $\dot{N}_{t_n}$ converges strongly in $L^2(0, 1)$.
    \end{lemma}
    \begin{proof}
        \thmendspace
        After perhaps discarding finitely many terms, we can assume that $\{t_n\}$ is bounded strictly away from zero, and so there exists some $C > 0$ such that for all $n$ and almost every $a, b$ with $a < b$ we have by \Cref{lem:goodtimes} that $$\dot{N}_{t_n}(b) - \dot{N}_{t_n}(a) \leq C (N_{t_n}(b) - N_{t_n}(a)).$$
        In particular, if we define $Z_n := \dot{N}_{t_n} - C \cdot N_{t_n}$ it is clear that all of the $Z_n$ are monotone on $(0, 1)$.
        And moreover, $\norm{Z_n}_{L^\infty} \leq (1 + C) M$. 
        Thus, $\sup_n \mathrm{TV}(Z_n) < \infty$, and as $N_{t_n} \in \mathcal{M}$ we arrive at $\sup_n \mathrm{TV}(\dot{N}_{t_n}) < \infty$. 
        By the compactness of the embedding $\mathrm{BV}(0, 1) \hookrightarrow L^1(0, 1)$, this implies that up to a subsequence $\dot{N}_{t_n}$ converges strongly in $L^1(0, 1)$. 
        By the $L^\infty$ bound on $\dot{N}_{t_n}$, this upgrades to $L^2$ convergence, i.e. $\dot{N}_{t_n}$ converges strongly in $L^2(0, 1)$.
    \end{proof}

    We will take limits of $t \mapsto N_{t_n + t}$ for $t_n \rightarrow \infty$. 
    To that end, for $s \in G$ and $t \geq s$ we put $$n_{s, t} := \dot{N}_s + \int_s^t \mathcal{F}[N_\tau] d\tau.$$

    Note that $n_{s, t}$ and $\dot{N}_t$ differ only in that $n_{s, t}$ evolves without the projections. It would be helpful to quantify how much of a difference this projection makes; to that end we have the following identity:
    \begin{lemma} \label{lem:dissipation}
        For $s, t \in G$ with $s < t$ we have $\dot{N}_t = \mathrm{Proj}_{\mathcal{K}_t} n_{s, t}$ and $$\norm{n_{s, t} - \dot{N}_t}^2 = 2 \left( \mathcal{H}(s) - \mathcal{H}(t) \right).$$
        This also implies $\mathcal{H}$ is nonincreasing on $G$, and that $\mathcal{H}$ is constant if and only if $n_{s, t} = \dot{N}_t$ for all such $s, t$.
    \end{lemma}
    \begin{proof}
        \thmendspace
        Write $\tilde{V}_t := V_0 + \int_0^t \mathcal{F}[N_\tau] d\tau$. In particular, note that $\dot{N}_t = \mathrm{Proj}_{\mathcal{K}_t} \tilde{V}_t$ and $\tilde{V}_t = \tilde{V}_s + \int_s^t \mathcal{F}[N_\tau] d\tau$.

        First, note that by stickiness (\cite[Proposition 4.1(iii)]{Hynd2020}) there is some Borel function $\psi$ such that $N_t = \psi \circ N_s$, and thus $\mathcal{K}_t \subset \mathcal{K}_s$. So $\mathrm{Proj}_{\mathcal{K}_t} = \mathrm{Proj}_{\mathcal{K}_t} \mathrm{Proj}_{\mathcal{K}_s}$. And thus $$\dot{N}_t = \mathrm{Proj}_{\mathcal{K}_t} \tilde{V}_t = \mathrm{Proj}_{\mathcal{K}_t} \left( \tilde{V}_s + \int_s^t \mathcal{F}[N_\tau] d \tau \right) = \mathrm{Proj}_{\mathcal{K}_t} \left( \dot{N}_s + \int_s^t \mathcal{F}[N_\tau] d \tau \right) = \mathrm{Proj}_{\mathcal{K}_t} n_{s, t}.$$
        
        By \Cref{lem:forcepotential}(2) and \eqref{eqn:lipschitzofN}, $\tau \mapsto n_{s,\tau}$ is $C^1$ with derivative $\mathcal{F}[N_\tau]$, and we get $\norm{n_{s, t}}^2 = \norm{\dot{N}_s}^2 + 2 \int_s^t \langle n_{s, \tau}, \mathcal{F}[N_\tau] \rangle d\tau$.
        For a.e. $\tau \in (s, t)$ we have $\tau \in G$, and $\mathcal{F}[N_\tau] = -(W' * \rho_\tau) \circ N_\tau$ which is clearly in $\mathcal{K}_\tau$. So by the first part of this lemma (applied to the pair $s < \tau$) and \Cref{lem:forcepotential}(3), $$\langle n_{s, \tau}, \mathcal{F}[N_\tau] \rangle = \langle \mathrm{Proj}_{\mathcal{K}_\tau} n_{s, \tau}, \mathcal{F}[N_\tau] \rangle = \langle \dot{N}_\tau, \mathcal{F}[N_\tau] \rangle = -\frac{d}{d\tau} \mathcal{P}[N_\tau].$$
        Since $\tau \mapsto \mathcal{P}[N_\tau]$ is Lipschitz (by \Cref{lem:forcepotential}(1) and \eqref{eqn:lipschitzofN}), the fundamental theorem of calculus gives $$\norm{n_{s, t}}^2 = \norm{\dot{N}_s}^2 - 2 (\mathcal{P}[N_t] - \mathcal{P}[N_s]).$$
        Finally, since $\dot{N}_t = \mathrm{Proj}_{\mathcal{K}_t} n_{s, t}$, the Pythagorean theorem gives
        \[\norm{n_{s, t} - \dot{N}_t}^2 = \norm{n_{s, t}}^2 - \norm{\dot{N}_t}^2 = 2 \left( \mathcal{H}(s) - \mathcal{H}(t) \right). \qedhere\]
    \end{proof}
    We are now in a position to prove the main result of this appendix. The main difference between this and the proof of \Cref{thm:existence} is that here we work at the level of trajectory maps (which are available in the $c$-convex setting), and here we do additional work to recover in the limit \emph{eternal} solutions on $\mathbb{R}$ (as opposed to just $[0, \infty)$) that solve a simpler PDE.
    \begin{proof}[Proof of \Cref{thm:shiftedlimits}]
        \thmendspace
        Fix a sequence $\{t_n\} \subset G$ with $t_n \rightarrow \infty$. 

        Recall that we require (by convention) that $W \geq 0$, which we can always do without loss of generality. And so $\mathcal{H} \geq 0$. Thus $\mathcal{H}_\infty = \inf_{t \in G} \mathcal{H}(t)$ exists and is finite.
        
        \textit{Step 1}. We will now construct $Y$. First, by construction of $G$ (and the selection of $\{t_n\} \subset G$) we have ensured that for every $k \in \mathbb{Z}_{\geq 0}$ we have $t_n - k \in G$ whenever $t_n > k$. And so by \Cref{lem:strongvelocityconvergence} and a diagonal argument, we can pass to a subsequence of $\{t_n\}$ (which we will suppress, i.e. still refer to as $\{t_n\}$) so that there exists, for each $k \in \mathbb{Z}_{\geq 0}$, functions $Z_k \in L^2(0, 1)$ such that $\dot{N}_{t_n - k} \rightarrow Z_k$ strongly in $L^2(0, 1)$.

        Now, recall the fact that $N_{t_n + t} \in \mathcal{M}$ and $\norm{N_{t_n + t} - N_{t_n + s}} \leq M |t - s|$ (where $M$ is defined in \Cref{eqn:appendixsupportbounds}). So by the Arzel\`a--Ascoli theorem for Lipschitz maps applied to $[-T, T]$ for each $T \in \mathbb{N}$, followed by a diagonal argument to send $T \rightarrow \infty$, there is a subsequence (which we will again not notate) and a $M$-Lipschitz $Y: \mathbb{R} \rightarrow \mathcal{M}$ such that $$N_{t_n + \bullet} \rightarrow Y \qquad \text{in } C([-T, T]; L^2(0, 1)) \quad \forall T.$$
        Note that Arzel\`a--Ascoli applies here because $\mathcal{M}$ is a compact subset of $L^2(0, 1)$.
        
        Going forward, we will exclusively work with the subsequence that has the properties outlined above.

        \textit{Step 2}. We will now show that $Y \in C^2(\mathbb{R}; L^2(0, 1))$, and moreover that $Y$ satisfies the equation $\ddot{Y}_t = \mathcal{F}[Y_t]$.
        By \Cref{lem:forcepotential}(2), we know that $\tau \mapsto \mathcal{F}[Y_\tau]$ is continuous and $\mathcal{F}[N_{t_n + \tau}] \rightarrow \mathcal{F}[Y_\tau]$ in $L^2(0, 1)$ for every $\tau$. 
        For any $k \geq 0$ and $t \geq -k$ put $$\tilde{n}^n_t := n_{t_n - k, t_n + t} = \dot{N}_{t_n - k} + \int_{-k}^t \mathcal{F}[N_{t_n + \tau}] d\tau \qquad \tilde{n}_t := Z_k + \int_{-k}^t \mathcal{F}[Y_\tau] d\tau.$$
        Note that $\dot{N}_{t_n - k} \rightarrow Z_k$ \emph{strongly} in $L^2(0, 1)$, and by combining this with dominated convergence for the integral term (which applies because $|\mathcal{F}[N_{t_n + \tau}]| \leq L$), we have $\tilde{n}^n_t \rightarrow \tilde{n}_t$ in $L^2(0, 1)$ for every $t \geq -k$.

        Now let $\mathcal{S} = \bigcup_m \bigcap_{n \geq m} (G - t_n)$. As we have only countably many $t_n$, this has full measure on $\mathbb{R}$.
        For $t \in \mathcal{S}$ with $t > -k$, we can apply \Cref{lem:dissipation} to get
        $$\norm{\dot{N}_{t_n + t} - \tilde{n}^n_t}^2 = 2 \left( \mathcal{H}(t_n - k) - \mathcal{H}(t_n + t) \right) \leq 2 \left( \mathcal{H}(t_n - k) - \mathcal{H}_\infty \right) \rightarrow 0,$$
        and so $\dot{N}_{t_n + t} \rightarrow \tilde{n}_t$ strongly in $L^2(0, 1)$. 
        Clearly, $N_{t_n + t} - N_{t_n - k} = \int_{-k}^t \dot{N}_{t_n + \tau} d\tau$ with $\norm{\dot{N}_{t_n + \tau}} \leq M$ for $\tau \in \mathcal{S}$ (by \eqref{eqn:lipschitzofN}), so another application of dominated convergence gives $$Y_t - Y_{-k} = \int_{-k}^t \tilde{n}_\tau d\tau \qquad \forall t \geq -k.$$
        By construction, $\tilde{n} \in C^1([-k, \infty); L^2(0, 1))$ with $\dot{\tilde{n}}_t = \mathcal{F}[Y_t]$. By earlier analysis, we have now shown that $$\dot{Y} = \tilde{n} \qquad \ddot{Y}_t = \mathcal{F}[Y_t] \text{ on } [-k, \infty).$$ 
        And as $k$ is arbitrary, $Y \in C^2(\mathbb{R}; L^2(0, 1))$ and $\ddot{Y}_t = \mathcal{F}[Y_t]$ for all $t \in \mathbb{R}$. 
        In particular, taking $k=0$ and $t=0$ gives $\dot{Y}_0 = Z_0 = \lim_{n \rightarrow \infty} \dot{N}_{t_n}$. Thus, we have shown that $Y$ is an eternal solution. 

        \textit{Step 3.} Finally, we show that $Y$ is constant-Hamiltonian solution. By \Cref{lem:forcepotential}(3) and the results of the earlier step, $$\frac{d}{dt} \left( \frac{1}{2} \norm{\dot{Y}_t}^2 + \mathcal{P}[Y_t] \right) = \langle \ddot{Y}_t, \dot{Y}_t \rangle - \langle \mathcal{F}[Y_t], \dot{Y}_t \rangle = 0.$$
        So the Hamiltonian of $Y$ is constant, and we just need to check that the energy at $t=0$ is equal to the limiting Hamiltonian along $\{t_n\}$. Well, because velocities at $t=0$ converge strongly, we can apply \Cref{lem:forcepotential}(1) to get: \[\frac{1}{2} \norm{\dot{Y}_0}^2 + \mathcal{P}[Y_0] = \lim_{n \rightarrow \infty} \left( \frac{1}{2} \norm{\dot{N}_{t_n}}^2 + \mathcal{P}[N_{t_n}] \right) = \lim_{n \rightarrow \infty} \mathcal{H}(t_n) = \mathcal{H}_\infty.\qedhere\]
    \end{proof}
    It may be interesting to use the above result to analyze limiting behavior of SPS. For example, the following result is immediate in light of previous work:
    \begin{corollary} \label{corr:convergencetodelta}
        Consider a potential $W$ and a trajectory map solution $(\rho_t, v_t, X(t, \cdot))$ compatible with \Cref{defn:trajectorymapsoln}. Assume additionally that $W'(x) > 0$ for all $x \in (0, \infty)$. Then $$W_2(\rho_t, \delta_0) \rightarrow 0 \qquad \text{as } t \rightarrow \infty.$$
    \end{corollary}
    \begin{proof}
        \thmendspace
        In light of \Cref{thm:shiftedlimits}, it suffices to just show that the only eternal solution $Y$ with $\int_0^1 Y_t da = 0$ for all $t$ is $Y \equiv 0$. From this, the limit of $\rho_t$ along any sequence of times in $G$ must converge to $\delta_0$. As $G$ is of full measure in $(0, \infty)$ it must also be dense, and as $t \mapsto N_t$ is Lipschitz this shows that limits of $\rho_t$ along any sequence of times converges to $\delta_0$.

        Now, let $Y$ be as above, and put $R(t) := 2 \int_{0.5}^1 Y_t da$, which is bounded and $C^2$ as $Y \in C^2(\mathbb{R}; L^2(0, 1))$. 
        Then because $\ddot{Y}_t = \mathcal{F}[Y_t]$ we can write (by oddness of $W'$) the following: $$R''(t) = -2 \int_{1/2}^1 \int_0^1 W'(Y_t(a) - Y_t(b)) \, db \, da = -2 \int_{1/2}^1 \int_0^{1/2} W'(Y_t(a) - Y_t(b)) \, db \, da \leq 0.$$
        Thus $R$ is concave and bounded on all of $\mathbb{R}$, and so $R$ is constant. In particular $R'' \equiv 0$, and as $W' > 0$ on $(0, \infty)$ this forces $Y_t(a) = Y_t(b)$ for almost every $a > 1/2$ and $b < 1/2$. So $Y_t$ is constant, hence $Y_t \equiv 0$ as it has mean zero.
    \end{proof}
    
    \begin{remark}
        There are fairly simple proofs for the above corollary by ODE comparison if $W$ is assumed to be convex, but to our knowledge none of these approaches can be easily generalized to the semiconvex setting. The key difficulty is controlling collisions (for example, the quantity $R(t)$ considered in the proof above loses concavity at the first collision time). However, in taking limits using \Cref{thm:shiftedlimits} we obtain eternal solutions for which there are no ``collisions''.
    \end{remark}

    At a higher level, \Cref{thm:shiftedlimits} provides an avenue to analyze the $\omega$-limit set of a SPS, by understanding the possible eternal solutions for each limiting energy value. In particular, stationary and periodic solutions are all eternal solutions. We will not pursue this angle here.

\section*{Acknowledgements}
This work was partially supported by the National Science Foundation under Grant Number DMS-2349040 (PI: C. Tsikkou).
	Any opinions, findings, and conclusions or recommendations expressed in this
material are those of the authors and do not necessarily reflect the views of the National Science Foundation.

\bibliographystyle{plain}
\bibliography{references}
\end{document}